\documentclass{article}

\usepackage{amsmath,amssymb,amsthm}
\usepackage{mathrsfs}
\usepackage{amscd}
\usepackage{enumerate}

\theoremstyle{plain}
\newtheorem{theorem}{Theorem}[section]
\theoremstyle{remark}

\newtheorem{definition}[theorem]{Definition}
\theoremstyle{plain}
\newtheorem{corollary}[theorem]{Corollary}
\newtheorem{lemma}[theorem]{Lemma}
\newtheorem{proposition}[theorem]{Proposition}

\numberwithin{equation}{section}

\def\N{{\mathbb N}}
\def\Z{{\mathbb Z}}

\def\R{{\mathbb R}}
\def\C{{\mathbb C}}

\title{Maximal and quadratic Gaussian Hardy spaces}

\author{Pierre Portal \footnote{This research is supported by the Australian Research Council Discovery grant DP120103692.}}

\begin{document}

\maketitle

\begin{abstract}
Building on the author's recent work with Jan Maas and Jan van Neerven, this paper establishes the equivalence of two norms (one using a maximal function, the other a square function) used to define a Hardy space on $\R^{n}$ with the gaussian measure, that is adapted to the Ornstein-Uhlenbeck semigroup. 
In contrast to the atomic Gaussian Hardy space introduced earlier by Mauceri and Meda, the $h^{1}(\R^{n};d\gamma)$ space studied here is such that the Riesz transforms are bounded from $h^{1}(\R^{n};d\gamma)$ to
$L^{1}(\R^{n};d\gamma)$.
This gives a gaussian analogue of the seminal work of Fefferman and Stein in the case of the Lebesgue measure and the usual Laplacian.
\end{abstract}

\section{Introduction}
In recent years, the real variable theory of Hardy spaces, which originates from the work of Fefferman and Stein \cite{fs}, has been extend to a variety of new settings.
These developments involve replacing the euclidean Laplacian with a different semigroup generator $L$, and the space $\R^{n}$ endowed with the Borel algebra and the Lebesgue measure with a different metric measure space $(M,d,\mu)$. Prominent examples include Hofmann and Mayboroda's work \cite{hm} on the euclidean space with $\Delta$ replaced by a more general divergence form second order elliptic differential operator with bounded measurable coefficients, and Auscher-McIntosh-Russ's Hardy  spaces of differential forms associated with the Hodge Laplacian on a Riemannian manifold \cite{amr}.
These results rely heavily on two assumptions. At the level of the metric measure space, one requires the doubling property:
there exists $C>0$ such that for all $x \in M$ and all $r>0$:
$$
\mu(B(x,2r)) \leq C \mu(B(x,r)).
$$
At the level of the semigroup $(e^{tL})_{t \geq 0}$, one requires some heat kernel estimates or, at least, some appropriate $L^{2}$ off-diagonal decay of the form
$$
\|1_{E}e^{tL}(1_{F}u)\|_{2} \leq C (1+\frac{d(E,F)^{2}}{t})^{-k}\|1_{F}u\|_{2},
$$
where $E,F$ are Borel sets, $1_{E},1_{F}$ denote the corresponding characteristic functions, $u \in L^{2}$, $k>0$, $t>0$ and $C$ is independent of $E,F,t$ and $u$.
This paper is concerned with the gaussian case: the metric measure space is $\R^{n}$ with the gaussian measure $d\gamma(x) = \pi^{-\frac{n}{2}}e^{-|x|^{2}}dx$ and the operator is the Ornstein-Uhlenbeck operator defined by
$$
Lf(x) := \frac{1}{2}\Delta f(x)-x.\nabla f(x), \quad x \in \R^{n}.
$$
This setting is motivated by stochastic analysis and has a long history (see the survey \cite{survey}).
Hardy spaces in this context were first introduced by Mauceri and Meda in \cite{mm}. Their work is striking because the gaussian measure is not doubling, 
and the Ornstein-Uhlenbeck semigroup does not satisfy the kernel bounds required to apply the non-doubling theory of Tolsa \cite{t}.
While \cite{mm} contains highly interesting results, it does not provide a fully satisfying theory. This is due to the fact that Mauceri-Meda's Hardy spaces 
$h^{1}_{\text{at}}(\gamma)$ are defined via an atomic decomposition that may not relate to the Ornstein-Uhlenbeck operator as well as classical Hardy spaces relate to the usual Laplacian (see \cite{fs}). In particular, the fact, proven in \cite{mms}, that some associated Riesz transforms are not bounded from $h^{1}_{\text{at}}(\gamma)$ to $L^{1}(\gamma)$ in dimension greater than $1$ is problematic. More generally, Mauceri-Meda's $h^{1}_{\text{at}}(\gamma)$ spaces provide a good endpoint to the $L^p$ scale from the interpolation point of view, but their theory does not contain all the machinery that makes Fefferman-Stein \cite{fs} so outstanding, and has proven useful in a range of applications, especially to partial differential equations.\\

In \cite{mnp,mnp2}, Jan Maas, Jan van Neerven, and the author have started the development of such a complete theory. 
This involves adequate dyadic cubes, covering lemmas of Whitney type, related tent spaces and their atomic decomposition, and techniques to estimate the following non-tangential maximal functions and conical square functions:
\begin{equation*}
\begin{split}
T^{*}_{a}u(x):=\underset{(y,t)\in \Gamma^{a}_{x}(\gamma)}{\sup}|e^{t^{2}L}u(y)|,\\
S_{a}u(x) =
 \Big(\int_{\Gamma^{a} _x(\gamma)} \frac1{\gamma(B(y,t))}|t\nabla
e^{t^{2}L}u(y)|^{2}\,d\gamma(y)\,\frac{dt}{t}\Big)^\frac12,
\end{split}
\end{equation*}
where 
$$\Gamma ^{a}_{x}(\gamma) := \Big\{(y,t) \in \R^{n}\times(0,\infty)\colon 
|y-x|<t<a m(x)\Big\}$$
is the {\em admissible cone} based at the point
$x\in\R^n$, $m(x):=\min\big\{1, \frac1{|x|}\big\}$ is the corresponding admissibility function, and $a$ the admissibility parameter.
From the point of view of Hardy space theory, one defines $h^{1}_{\text{max},a}(\gamma)$ as the completion of the space of smooth compactly supported functions 
$C^{\infty}_{c}(\R^{n})$ with respect to $$\|u\|_{h^{1}_{\text{max},a}(\gamma)}:=\|T_{a} ^{*}u\|_{L^{1}(\gamma)},$$ and $h^{1}_{\text{quad},a}(\gamma)$ as the completion of $C^{\infty}_{c}(\R^{n})$
with respect to $$\|u\|_{h^{1}_{\text{quad},a}(\gamma)}:=\|S_{a}u\|_{L^{1}(\gamma)}+\|u\|_{L^{1}(\gamma)}.$$
A key result should be that these two norms are equivalent for some choice of $a$. However, \cite{mnp2} only gives one inequality: 
$\|S_{a}u\|_{1} \leq C\|T_{a'}^{*}u\|_{1}$, for some $C,a'>0$ independent of $u$ (actually \cite{mnp2} gives a slightly stronger inequality involving an averaged version of $T^{*}_{a}u$). 
The purpose of this paper is to prove the reverse inequality to establish the following result.
\begin{theorem}
\label{thm:main}
Given $a>0$, there exists $a'>0$ such that
$h^{1}_{\text{quad},a}(\gamma)=h^{1}_{\text{max},a'}(\gamma)$.
\end{theorem}
Since $h^{1}_{\text{quad},a} = h^{1}_{\text{quad},1}$ for all $a>1$ (as a consequence of \cite[Theorem 3.8]{mnp}), we then call $h^{1}(\gamma):=h^{1}_{\text{quad},2}$ the Gaussian Hardy space.
In the final section, the techniques used in the proof of the above reverse inequality are used again to prove that the Riesz transforms associated with $L$ are bounded on $h^{1}(\gamma)$.
The proof is based on a version of Calder\'on reproducing formula:
$$
u = C \int \limits _{0} ^{\infty} (t^{2}L)^{N+1}e^{\frac{(1+a^{2})t^{2}}{\alpha}L}u \frac{dt}{t}+\int \limits _{\R^{n}}u d\gamma,
$$
for $u \in L^{2}$ and some suitable constants $N,C$ and $\alpha$. The part 
$$J_{1}u(x):=\int \limits _{0} ^{m(x)} (t^{2}L)^{N+1}e^{\frac{(1+a^{2})t^{2}}{\alpha}L}u(x) \frac{dt}{t}$$ 
is treated via the atomic decomposition of tent spaces established in \cite{mnp}, leading to the estimate 
$\|J_{1}u\|_{h^{1}_{\text{max},a'}(\gamma)} \leq C'(\|u\|_{h^{1}_{\text{quad},a}(\gamma)}+\|u\|_{L^{1}(\gamma)})$.
The remainder term
$$J_{\infty}u(x):=\int \limits _{m(x)} ^{\infty} (t^{2}L)^{N+1}e^{\frac{(1+a^{2})t^{2}}{\alpha}L}u(x) \frac{dt}{t}$$ 
is a priori problematic, as the boundedness of the square function norm $\|S_{a}u\|_{1}$ does not give information about it.
It turns out, however, that properties of the kernel of the Ornstein-Uhlenbeck semigroup give the estimate
$\|J_{\infty}u\|_{h^{1}_{\text{max},a'}(\gamma)} \leq C''\|u\|_{L^{1}(\gamma)}$. This phenomenon is typical of local Hardy spaces, as can be seen, for instance, in \cite{art} and \cite{jyz}.\\

The paper is organised as follows. In Section 2, we recall the necessary definitions and known results, and set up the proof,
decomposing $J_{1}u$ into a main term and two remainder terms similar to $J_{\infty}u$.
In Section 3, we prove the relevant kernel estimates, and deduce appropriate off-diagonal bounds.
In Section 4, we show that the main term can be decomposed as a sum of molecules,
and estimate the $h^{1}_{\text{max}}$ norm of molecules.
In Section 5, we estimate $J_{\infty}u$ and the remainder terms, and thus conclude the proof.
In Section 6, we use the same techniques to prove that the Riesz transforms associated with $L$ are bounded on $h^{1}(\gamma)$.

\subsection*{Acknowledgement}
This work completes the first part of a larger project, started in \cite{mnp,mnp2} in collaboration with Jan Maas and Jan van Neerven. It owes a great deal to our discussions.
Thanks to Adam Sikora, for correcting an incorrect comment in an earlier version of the  introduction.
Many thanks also go to Jonas Teuwen, who read a draft of this paper as part of his master thesis, under the supervision of Jan van Neerven at the TU Delft. Numerous misprints and omissions have been fixed thanks to his careful reading.

\section{Preliminaries}
We start by recalling some basic properties of the Ornstein-Uhlenbeck operator $L$ (details can be found in the survey paper \cite{survey}).
On $L^{2}(\gamma)$, $L$ generates a semigroup for which the Hermite polynomials $(H_{\alpha})_{\alpha \in \Z_{+}^{n}}$ form an orthonormal basis of eigenfunctions. Using this chaos decomposition, we have:
$$
e^{tL}(\sum \limits _{\beta \in \Z_{+}^{n}} c_{\beta}H_{\beta}) = \sum \limits _{\beta \in \Z_{+}^{n}} e^{-t|\beta|}c_{\beta}H_{\beta},
$$  
for $c_{\beta} \in \C$ and $|\beta|:= \sum \limits _{j=1} ^{n} \beta_{j}$.
As a direct consequence, we have the following Calder\'on reproducing formula.
\begin{lemma}
\label{lem:repro}
For all $N \in \N$ and $a,\alpha>0$, there exists $C>0$ such that for all $u \in L^{2}(\gamma)$
$$
u = C \int \limits _{0} ^{\infty} (t^{2}L)^{N+1}e^{\frac{(1+a^{2})t^{2}}{\alpha}L}u \frac{dt}{t}+\int \limits _{\R^{n}}u d\gamma.
$$
\end{lemma}
On $L^{p}(\R^{n},\gamma)$, for $1\leq p<\infty$, $L$ generates the semigroup defined by
$$
e^{tL}f(x) := \int \limits _{\R^{n}} M_{t}(x,y)f(y)dy,
$$
where $f \in L^{p}(\gamma)$, $x \in \R^{n}$, and $M_{t}$ denotes the Mehler kernel
$$
M_{t}(x,y) := \pi^{-\frac{n}{2}}(1-e^{-2t})^{-\frac{n}{2}}\exp(-\frac{|e^{-t}x-y|^{2}}{1-e^{-2t}}).
$$
A well know technique in gaussian harmonic analysis, going back to \cite{m}, consists in splitting kernels such as the Mehler kernel into a local and a global part, the idea being that the local part behaves like a Caldero\'n-Zygmund operator, and the global part has some specific decay properties. The local region is defined as
$$
N_{a} := \{(x,y) \in \R^{2n} \;;\; |x-y| \leq a m(x)\},
$$
where $a>0$ and  $m(x):=\min\big\{1, \frac1{|x|}\big\}$.
A typical result obtained by this technique is the weak-type 1-1 of the local part of the Hardy-Littlewood maximal operator and the strong type 1-1 of its global part, proven by Harboure, Torrea, and Vivani in \cite[Theorem 2.7]{htv}.
In this paper, we will use the corresponding result for the non-tangential maximal function.
Before stating this result, we recall \cite[Lemma 2.3]{mnp}, and introduce some notation.
\begin{lemma}
\label{lem:mnp1}
Let $a>0$, and $x,y \in \R^{n}$.
If $|x-y|<am(x)$, then $m(x)\le (1+a)m(y)$ and $m(y)\le (2+2a)m(x)$.
\end{lemma}
Given $A,a>0$, we define
$$\Gamma^{(A,a)}_{x}(\gamma) := \Big\{(y,t) \in \R^{n}\times(0,\infty)\colon 
|y-x|<At, \;\text{and}\; t\leq am(x)\Big\},$$
and call $\Gamma^{(A,a)}_{x}(\gamma)$ the admissible cone with aperture $A$ and admissibility parameter $a$ based at the point $x$.
To simplify notation we write $\Gamma_{x}(\gamma):=\Gamma^{(1,1)}_{x}(\gamma)$ and $\Gamma^{a}_{x}(\gamma):=\Gamma^{(1,a)}_{x}(\gamma)$.
Non-tangential maximal functions are pointwise dominated by the Hardy-Littlewood maximal function. This is the following lemma, proven by Pineda and Urbina in \cite[Lemma 1.1]{pu} (for the particular choice $(A,a)=(1,\frac{1}{2})$, but the proof carries over to different apertures and admissibility parameters).
\begin{lemma} 
\label{lem:pu}
Let $A,a>0$. There exists $C>0$ such that for all $x \in \R^{n}$ and all $f \in L^{2}(\gamma)$
$$
\underset{(y,t) \in \Gamma_{x} ^{(A,a)}(\gamma)}{\sup} |e^{t^{2}L}f(y)| \leq C \underset{r>0}{\sup} \frac{1}{\gamma(B(x,r))} \int \limits _{B(x,r)} |f(z)|d\gamma(z).
$$
\end{lemma}
Using \cite[Theorem 2.7]{htv}, we get the $L^{2}$ boundedness of non-tangential maximal functions, and the $L^{1}$ boundedness of their global parts.
\begin{proposition}
\label{prop:glob}
Let $A,a>$ and set $\tau := \frac{(1+aA)(1+2aA)}{2}$. Then, for $f \in C_{c}^{\infty}(\R^{n})$,
\begin{enumerate}[(i)]
\item
$$
\|T^{*}_{glob,a,A}f:x \mapsto \underset{(y,t) \in \Gamma_{x} ^{(A,a)}(\gamma)}{\sup} \int \limits _{\R^{n}} M_{t^{2}}(y,z)1_{N_{\tau}^{c}}(y,z)|f(z)|dz\|_{1} \lesssim  \|f\|_{1}.
$$
\item
$$
\|x \mapsto \underset{(y,t) \in \Gamma_{x} ^{(A,a)}(\gamma)}{\sup} \int \limits _{\R^{n}} M_{t^{2}}(y,z)|f(z)|dz\|_{2} \lesssim  \|f\|_{2}.
$$
\end{enumerate}
\end{proposition}
Here,  $\|x \mapsto \underset{(y,t) \in \Gamma_{x} ^{(A,a)}(\gamma)}{\sup} \int \limits _{\R^{n}} M_{t^{2}}(y,z)1_{N_{\tau}^{c}}(y,z)|f(z)|dz\|_{1} \lesssim  \|f\|_{1}$ means  
$$\|x \mapsto \underset{(y,t) \in \Gamma_{x} ^{(A,a)}(\gamma)}{\sup} \int \limits _{\R^{n}} M_{t^{2}}(y,z)1_{N_{\tau}^{c}}(y,z)|f(z)|dz\|_{1} \leq C \|f\|_{1}$$ for some $C>0$ independent of $f$. We will use this notation throughout the paper.
\begin{proof}
For $x\in \R^{n}$, $(y,z) \in N_{\tau} ^{c}$, and $(y,t)  \in \Gamma_{x} ^{(A,a)}(\gamma)$, we have that
$$
|x-z| \geq \tau m(y) -aAm(x) \geq (\frac{\tau}{1+aA}-aA)m(x)=\frac{1}{2}m(x). 
$$
Therefore
$$
\|x \mapsto \underset{(y,t) \in \Gamma_{x} ^{(A,a)}(\gamma)}{\sup} \int \limits _{\R^{n}} M_{t^{2}}(y,z)1_{N_{\tau}^{c}}(y,z)|f(z)|dz\|_{1} \leq
\|x \mapsto \underset{(y,t) \in \Gamma_{x} ^{(A,a)}(\gamma)}{\sup} \int \limits _{\R^{n}} M_{t^{2}}(y,z)g_{x}(z)dz\|_{1},
$$
where $g_{x}(z):=1_{N_{\frac{1}{2}}^{c}}(x,z)|f(z)|$.
Lemma \ref{lem:pu}, combined with \cite[Theorem 2.7]{htv} thus gives
$$
\|x \mapsto \underset{(y,t) \in \Gamma_{x} ^{(A,a)}(\gamma)}{\sup} \int \limits _{\R^{n}} M_{t^{2}}(y,z)1_{N_{\tau}^{c}}(y,z)|f(z)|dz\|_{1} \lesssim 
\int \limits _{\R^{n}} \underset{r>0}{\sup} \frac{1}{\gamma(B(x,r))} \int \limits _{B(x,r)}1_{N_{\frac{1}{2}}^{c}}(x,z) |f(z)|d\gamma(z) \lesssim \|f\|_{1}.
$$
To prove (ii), we apply Lemma \ref{lem:pu} and Lemma \ref{lem:mnp1} to obtain, for $x\in \R^{n}$,
$$
 \underset{(y,t) \in \Gamma_{x} ^{(A,a)}(\gamma)}{\sup} \int \limits _{\R^{n}} 1_{N_{\tau}}(y,z)M_{t^{2}}(y,z)|f(z)|dz \lesssim
 \underset{r \in (0,\tau' m(x))}{\sup} \frac{1}{\gamma(B(x,r))} \int \limits _{B(x,r)} |f(z)| d\gamma(z),
$$
for $\tau' = aA+\tau(2+2aA)$ and an implicit constant independent of $x$.
The weak type $1-1$ of this local part is proven, for instance, in \cite[Lemma 3.2]{mnp}. Combined with (i), this gives the weak type $1-1$ of
$$x \mapsto \underset{(y,t) \in \Gamma_{x} ^{(A,a)}(\gamma)}{\sup} \int \limits _{\R^{n}} M_{t^{2}}(y,z)|f(z)|dz.$$
Given the (obvious) $L^{\infty}$ boundedness of the Hardy-Littlewood maximal function
(and thus of the non-tangential maximal function by Lemma \ref{lem:pu}), the proof follows by interpolation.
\end{proof}

A geometric version of the local/global dichotomy is given by the key notion of admissible balls, introduced in \cite{mm}.
Defining
$$
\mathcal{B}_{a} := \{B(x,r) \;;\; x \in \R^{n}, \quad 0<r \leq am(x)\},
$$
we say that a ball $B \in \mathcal{B}_{a}$ is admissible at scale $a$.
The gaussian measure acts as a doubling measure on admissible balls, as Mauceri and Meda have pointed out in \cite[Proposition 2.1]{mm}.
We recall here a version of their result.
\begin{lemma}
\label{lem:mm}
There exists $C>0$ such that for all $a,b\geq 1$ and all $B(x,r) \in \mathcal{B}_{a}$ we have
$$
\gamma(B(x,br)) \leq e^{2a^{2}(2b+1)^{2}}\gamma(B(x,r)).
$$
\end{lemma}
This led Jan Maas, Jan van Neerven and the author to introduce gaussian tent spaces, in \cite{mnp}, as follows.
Let $D:= \{(t,x)\in (0,\infty) \times \R^{n} \;;\; t<m(x)\}$. Then $t^{1,2}(\gamma)$ is the completion of $C_{c}(D)$ with respect to the norm
$$
\|F\|_{t^{1,2}(\gamma)}:= \int \limits _{\R^{n}}  \Big(\int_{\Gamma_x(\gamma)} \frac1{\gamma(B(y,t))}|F(t,y)|^{2}\,d\gamma(y)\,\frac{dt}{t}\Big)^\frac12 d\gamma(x).
$$
Compared to \cite{mnp}, we are using here the notation $t^{1,2}(\gamma)$ rather than $T^{1,2}(\gamma)$ to emphasise the local nature of this space.
Theorem 3.4 in \cite{mnp} gives an atomic decomposition of $t^{1,2}(\gamma)$. Given $a>0$, a function $F:D\to\C$ is called a $t^{1,2}(\gamma)$ a-atom if
there exists a ball $B \in \mathcal{B}_{a}$ such that $supp(F) \subset \{(t,y)\in (0,\infty) \times \R^{n} \;;\; t\leq \min(d(y,B^{c}),m(y))\}$ and
$$
\int \limits _{\R^{n}} \int \limits _{0} ^{\infty} |F(t,y)|^{2} \frac{dydt}{t} \leq \gamma(B)^{-1}.
$$
\begin{theorem}
\label{thm:atomic}
For all $f \in t^{1,2}(\gamma)$ and $a>1$, there exists a sequence 
$(\lambda_{n})_{n\geq 1}
\in \ell_{1}$ and a sequence of $t^{1,2}(\gamma)$ $a$-atoms $(F_{n})_{n\geq 1}$
such that
\begin{enumerate}
\item[\rm(i)] $f = \sum  _{n\geq 1}\lambda_{n}F_{n}$;
\item[\rm(ii)]
$\sum  _{n\geq 1}|\lambda_{n}| \lesssim
\|f\|_{t^{1,2}(\gamma)}$.
\end{enumerate}
\end{theorem}
To simplify notation we will simply call atoms the $t^{1,2}(\gamma)$ 2-atoms.
Combining the atomic decomposition of $t^{1,2}(\gamma)$ and Lemma \ref{lem:repro} we get the following decomposition, which is the basis of the proof of Theorem \ref{thm:main}.
\begin{corollary}
\label{cor:dec}
For all $N \in \N$, $a>1$, $b>0$, and $\alpha>a^{2}$, there exists $C>0$ and $n$ sequences of atoms  $(F_{m,j})_{m\in \N}$ and complex numbers 
$(\lambda_{m,j})_{m\in\N}$ for $j=1,...n$, such that for all 
$u \in C_{c}^{\infty}(\R^{n})$ and $x\in \R^{n}$:
\begin{equation*}
\begin{split}
u(x) = \int \limits _{\R^{n}} u d\gamma
&- C \sum \limits _{j=1} ^{n} \sum \limits _{m=1} ^{\infty} \lambda_{m,j} \int \limits _{0} ^{2} (t^{2}L)^{N}e^{\frac{t^{2}}{\alpha}L}t\partial_{x_{j}}^{*}F_{m,j}(t,x)\frac{dt}{t}\\
&+C \sum \limits _{j=1} ^{n}\sum \limits _{m=1} ^{\infty} \lambda_{m,j} \int \limits _{0}^{2} 1_{[\frac{m(x)}{b},2]}(t) (t^{2}L)^{N}e^{\frac{t^{2}}{\alpha}L}t\partial_{x_{j}}^{*}F_{m,j}(t,x)\frac{dt}{t}  
\\&-C \sum \limits _{j=1} ^{n} \int \limits _{0} ^{\frac{m(x)}{b}} (t^{2}L)^{N}e^{\frac{t^{2}}{\alpha}L}t\partial_{x_{j}}^{*}(1_{D^{c}}(t,.)t\partial_{x_{j}}e^{\frac{a^{2}t^{2}}{\alpha}L})u(x)\frac{dt}{t}- C \int \limits _{\frac{m(x)}{b}} ^{\infty} (t^{2}L)^{N+1}e^{\frac{(1+a^{2})t^{2}}{\alpha}L}u(x) \frac{dt}{t},
\end{split}
\end{equation*}
and $ \sum \limits _{j=1} ^{n} \sum \limits _{m=1} ^{\infty} |\lambda_{m,j}| \lesssim \|u\|_{h^{1}_{\text{quad},a}}$. 
\end{corollary}
Here $\partial_{x_{j}}^{*}$ denotes the adjoint of $\partial_{x_{j}}$ in $L^{2}(\gamma)$.
\begin{proof}
We first remark that
$$
(t^{2}L)^{N+1}e^{\frac{(1+a^{2})t^{2}}{\alpha}L} = -\frac{1}{2} \sum \limits _{j=1} ^{n} (t^{2}L)^{N}e^{\frac{t^{2}}{\alpha}L}
t\partial_{x_{j}}^{*}((1_{D}(t,.)+1_{D^{c}}(t,.))t\partial_{x_{j}}e^{\frac{a^{2}t^{2}}{\alpha}L}u).
$$
It remains to check that the terms $1_{D}(t,.)t\partial_{x_{j}}e^{\frac{a^{2}t^{2}}{\alpha}L}u$, for $j\in \{1,...,n\}$, belong to $t^{1,2}(\gamma)$.
Using Lemma \ref{lem:mm} we have
\begin{equation*}
\begin{split}
\|(t,x) \mapsto 1_{D}(t,x)t\partial_{x_{j}}e^{\frac{a^{2}t^{2}}{\alpha}L}u(x)\|_{t^{1,2}(\gamma)}
&\lesssim \int \limits_{\R^{n}} ( \int \limits _{0} ^{\frac{m(x)}{\sqrt{\alpha}}} \int \limits _{B(x,\sqrt{\alpha}s)} 
\frac{1_{D}(\sqrt{\alpha}s,y)}{\gamma(B(y,\sqrt{\alpha}s))} |s\nabla e^{a^{2}s^{2}L}u(y)|^{2} d\gamma(y)\frac{ds}{s})^{\frac{1}{2}}d\gamma(x)\\
&\lesssim \int \limits_{\R^{n}} ( \int \limits _{0} ^{m(x)} \int \limits _{B(x,\sqrt{\alpha}s)} 
\frac{1_{D}(as,y)}{\gamma(B(y,s))} |s\nabla e^{a^{2}s^{2}L}u(y)|^{2} d\gamma(y)\frac{ds}{s})^{\frac{1}{2}}d\gamma(x).
\end{split}
\end{equation*}
By \cite[Theorem 3.8]{mnp}, we thus have
\begin{equation*}
\begin{split}
\|(t,x) \mapsto &1_{D}(t,x)t\partial_{x_{j}}e^{\frac{a^{2}t^{2}}{\alpha}L}u(x)\|_{t^{1,2}(\gamma)}
\lesssim \int \limits_{\R^{n}} ( \int \limits _{0} ^{m(x)} \int \limits _{B(x,as)} 
\frac{1_{D}(as,y)}{\gamma(B(y,s))} |s\nabla e^{a^{2}s^{2}L}u(y)|^{2} d\gamma(y)\frac{ds}{s})^{\frac{1}{2}}d\gamma(x)\\
&\lesssim \int \limits_{\R^{n}} ( \int \limits _{0} ^{am(x)} \int \limits _{B(x,t)} 
\frac{1_{D}(t,y)}{\gamma(B(y,t))} |t\nabla e^{t^{2}L}u(y)|^{2} d\gamma(y)\frac{dt}{t})^{\frac{1}{2}}d\gamma(x)
= \|u\|_{h^{1}_{\text{quad},a}}.
\end{split}
\end{equation*}
\end{proof}
Theorem \ref{thm:main} is then proven by combining the results in the next sections as follows.\\
{\em Proof of Theorem \ref{thm:main}:}\\
For $a>0$, \cite[Theorem 1.1]{mnp2} gives that there exists $a'>0$ such that $h^{1}_{\text{max},a'}(\gamma) \subset h^{1}_{\text{quad},a}(\gamma)$.
Let us fix this $a'$ and pick $\alpha>\max(2^{38},32e^{4},4\sqrt{a}e^{2a^{2}})$, $b \geq \max(2e,\sqrt{\frac{32e^{4}}{(\alpha-32e^{4})(1-e^{-2\frac{a^{2}}{\alpha}})}})$,
  and $N> \frac{n}{4}$.
Let $u \in C_{c}^{\infty}(\R^{n})$ and apply Corollary \ref{cor:dec}.
We have that
\begin{equation*}
\begin{split}
\|u\|_{h^{1}_{\text{max},a'}(\gamma)}
\lesssim 
\|T^{*}_{a'}(\int \limits _{\R^{n}} u d\gamma)\|_{1}
&+ C \sum \limits _{j=1} ^{n} \sum \limits _{m=1} ^{\infty} |\lambda_{m,j}|
\| \int \limits _{0} ^{2} (t^{2}L)^{N}e^{\frac{t^{2}}{\alpha}L}t\partial_{x_{j}}^{*}F_{m,j}(t,.)\frac{dt}{t}\|_{h^{1}_{\text{max},a'}(\gamma)}\\
&+C \sum \limits _{j=1} ^{n}\sum \limits _{m=1} ^{\infty} |\lambda_{m,j}|
\| \int \limits _{0}^{2} 1_{[\frac{m(.)}{b},2]}(t) (t^{2}L)^{N}e^{\frac{t^{2}}{\alpha}L}t\partial_{x_{j}}^{*}F_{n,j}(t,.)\frac{dt}{t}\|_{h^{1}_{\text{max},a'}(\gamma)} 
\\&+C \sum \limits _{j=1} ^{n} 
\|\int \limits _{0} ^{\frac{m(.)}{b}} (t^{2}L)^{N}e^{\frac{t^{2}}{\alpha}L}t\partial_{x_{j}}^{*}(1_{D^{c}}(t,.)t\partial_{x_{j}}e^{\frac{a^{2}t^{2}}{\alpha}L})u\frac{dt}{t}\|_{h^{1}_{\text{max},a'}(\gamma)}
\\&+ C \|\int \limits _{\frac{m(.)}{b}} ^{\infty} (t^{2}L)^{N+1}e^{\frac{(1+a^{2})t^{2}}{\alpha}L}u \frac{dt}{t}\|_{h^{1}_{\text{max},a'}(\gamma)} + \|u\|_{L^{1}(\gamma)}.
\end{split}
\end{equation*}
Since $e^{sL}1=1$ for all $s \geq0$, we have
$$
\|T^{*}_{a'}(\int ud\gamma)\|_{1} \leq \|u\|_{1} \leq \|u\|_{h^{1}_{\text{quad},a}(\gamma)}.
$$
Proposition \ref{prop:jinf} gives that 
$$
\|\int \limits _{\frac{m(.)}{b}} ^{\infty} (t^{2}L)^{N+1}e^{\frac{(1+a^{2})t^{2}}{\alpha}L}u \frac{dt}{t}\|_{h^{1}_{\text{max},a'}(\gamma)} \lesssim \|u\|_{1} \leq \|u\|_{h^{1}_{\text{quad},a}(\gamma)}.
$$
For $j \in \{1,...,n\}$, Proposition \ref{prop:Dcomp} then gives
$$
\|\int \limits _{0} ^{\frac{m(.)}{b}} (t^{2}L)^{N}e^{\frac{t^{2}}{\alpha}L}t\partial_{x_{j}}^{*}(1_{D^{c}}(t,.)t\partial_{x_{j}}e^{\frac{a^{2}t^{2}}{\alpha}L})u\frac{dt}{t}\|_{h^{1}_{\text{max},a'}(\gamma)}\lesssim \|u\|_{1} \leq \|u\|_{h^{1}_{\text{quad},a}(\gamma)}.
$$
Proposition \ref{prop:r1} gives that
$$
\| \int \limits _{0}^{2} 1_{[\frac{m(.)}{b},2]}(t) (t^{2}L)^{N}e^{\frac{t^{2}}{\alpha}L}t\partial_{x_{j}}^{*}F_{n,j}(t,.)\frac{dt}{t}\|_{h^{1}_{\text{max},a'}}
\lesssim 1, 
 $$
 while Proposition \ref{prop:mol} combined with Theorem \ref{thm:mol} gives
  $$
 \| \int \limits _{0} ^{2} (t^{2}L)^{N}e^{\frac{t^{2}}{\alpha}L}t\partial_{x_{j}}^{*}F_{n,j}(t,.)\frac{dt}{t}\|_{h^{1}_{\text{max},a'}(\gamma)}\lesssim 1.$$
Therefore 
$$
\|u\|_{h^{1}_{\text{max},a'}(\gamma)} \lesssim \|u\|_{h^{1}_{\text{quad},a}(\gamma)} + \sum \limits _{j=1} ^{n} \sum \limits _{m=1} ^{\infty} |\lambda_{m,j}|
\lesssim \|u\|_{h^{1}_{\text{quad},a}(\gamma)}.
$$
\section{Kernel estimates}
In this section, we establish some properties of the Mehler kernel, and use them to prove the following off-diagonal decay result.
Given $a>0$, $B=B(c_{B},r_{B}) \in \mathcal{B}_{a}$ and $k \in \Z_{+}$ we consider the following sets.
$$
C_{k}(B):= \begin{cases} B(c_{B},2r_{B}) \; \text{if} \; k=0, \\ B(c_{B},2^{k+1}r_{B}) \backslash B(c_{B},2^{k}r_{B}) \; \text{otherwise}.\end{cases}
$$
\begin{lemma}[Off-diagonal estimates]
Let $N \in \Z_{+}$, $a>0$, $j\in \{1,...,n\}$, $B \in \mathcal{B}_{a}$, $\alpha\geq 4e^{2a^{2}}$, and $k\in \N$. Then for all $u \in L^{2}(\gamma)$
$$
\|1_{C_{k}(B)}1_{(0,r_{B})}(t)(t^{2N+1}L^{N}e^{\frac{t^{2}}{\alpha}L}\partial_{x_{j}}^{*})1_{B}u\|_{2} \lesssim
\exp(-\frac{\alpha}{2^{6}e^{2a^{2}}}4^{k}(\frac{r_{B}}{t})^{2})\|u\|_{2},
$$
with implied constant depending only on $\alpha$ and $N$.
\end{lemma}
The above lemma plays a key role in the next sections, and could be deduced from more general methods giving $L^2$ off-diagonal bounds 
(see \cite{cs} or \cite{mcm}). We prove it through direct kernel estimates which are used in various parts of the paper. In the next sections, it will become clear that one needs off-diagonal decay of the form $\exp(-c4^{k})$ with $c$ large enough to compensate for the growth in Lemma \ref{lem:mm}. This is the reason why we use $e^{\frac{(1+a^{2})t^{2}}{\alpha}L}$ in the reproducing formula and pick $\alpha$ large enough.\\

Given $t,\alpha>0$, $j \in \{1,...,n\}$, and $N \in \Z_{+}$, we denote by $K_{t^{2},N,\alpha}$ and $\tilde{K}_{t^{2},N,\alpha,j}$ the relevant kernels defined, given $u \in L^{2}(\gamma)$, by
\begin{equation*}
\begin{split}
\int \limits _{\R^{n}} K_{t^{2},N,\alpha}(x,y)u(y)dy &= (t^{2}L)^{N}e^{\frac{t^{2}}{\alpha}L}u(x),\\
\int \limits _{\R^{n}} \tilde{K}_{t^{2},N,\alpha,j}(x,y)u(y)dy &= (t^{2}L)^{N}e^{\frac{t^{2}}{\alpha}L}t\partial_{x_{j}}^{*}u(x).
\end{split}
\end{equation*}
Note that $K_{t^{2},N,\alpha}(x,y) = t^{2N}\partial_{s}^{N}M_{s}(x,y)_{|s=\frac{t^{2}}{\alpha}}$, and that, by duality
$$\tilde{K}_{t^{2},N,\alpha,j}(x,y) = t^{2N+1}\partial_{y_{j}}\partial_{s}^{N}M_{s}(y,x)_{|s=\frac{t^{2}}{\alpha}}\exp(|x|^{2}-|y|^{2}).$$
To prove Lemma \ref{lem:od}, we need preparatory lemmas of independent interest.
\begin{lemma}
\label{lem:K}
Let $N \in \Z_{+}$. There exists $C_{N}\in \N$ and a polynomial of $2n+1$ variables $P_{N}$ of degree $C_{N}$ such that
for all $x,y \in \R^{n}$ and $s>0$:
$$
\partial_{s}^{N}M_{s}(x,y) = (1-e^{-2s})^{-N}P_{N}(e^{-s},(\frac{e^{-s}x_{j}-y_{j}}{\sqrt{1-e^{-2s}}})_{j=1,...,n}  , (\sqrt{1-e^{-2s}}x_{j})_{j=1,...,n})M_{s}(x,y).
$$
\end{lemma}
\begin{proof}
Let $j\in\{1,...,n\}$, $s>0$, $x,y \in \R^{n}$. We have the following.
\begin{equation*}
\begin{split}
\partial_{s}(\frac{e^{-s}x_{j}-y_{j}}{\sqrt{1-e^{-2s}}}) &= -(1-e^{-2s})^{-1}(e^{-s}x_{j}\sqrt{1-e^{-2s}}+e^{-2s}\frac{e^{-s}x_{j}-y_{j}}{\sqrt{1-e^{-2s}}}).\\
\partial_{s}(\sqrt{1-e^{-2s}}x_{j}) &= (1-e^{-2s})^{-1}(e^{-2s}\sqrt{1-e^{-2s}}x_{j}).\\
\partial_{s}M_{s}(x,y) &= -(1-e^{-2s})^{-1}ne^{-2s}M_{s}(x,y)-M_{s}(x,y)\partial_{s}(\frac{|e^{-s}x-y|^{2}}{1-e^{-2s}}).\\
\partial_{s}(\frac{(e^{-s}x_{j}-y_{j})^{2}}{1-e^{-2s}}) &= -(1-e^{-2s})^{-1}((2e^{-s}\sqrt{1-e^{-2s}}x_{j})(\frac{e^{-s}x_{j}-y_{j}}{\sqrt{1-e^{-2s}}})
+(\frac{e^{-s}x_{j}-y_{j}}{\sqrt{1-e^{-2s}}})^{2}2e^{-2s}).
\end{split}
\end{equation*}
The proof thus follows by induction. 
\end{proof}
Computing partial derivatives in $x_{j}$ one obtains in the same way:
\begin{corollary}
\label{cor:Ktilde}
Let $N \in \Z_{+}$ and $j\in\{1,...,n\}$. There exists $C_{N}\in \N$ and a polynomial of $2n+1$ variables $Q_{N}$ of degree $C_{N}$ such that
for all $x,y \in \R^{n}$ and $s>0$:
$$
\partial _{x_{j}}\partial_{s}^{N}M_{s}(x,y) = (1-e^{-2s})^{-(N+\frac{1}{2})}Q_{N}(e^{-s},(\frac{e^{-s}x_{j}-y_{j}}{\sqrt{1-e^{-2s}}})_{j=1,...,n}  , (\sqrt{1-e^{-2s}}x_{j})_{j=1,...,n})M_{s}(x,y).
$$
\end{corollary}

\begin{lemma} \label{lem:slow2}
For $a,C>0$, $\alpha >1,$ $t \in (0,a],$ and $x,y \in \R^n$ we have
\begin{enumerate}[(i)]
\item
$\exp(-C\frac{|e^{-\frac{t^{2}}{\alpha}}x-y|^{2}}{1-e^{-2\frac{t^{2}}{\alpha}}})
\leq \exp(-C\frac{\alpha}{2e^{2a^{2}}}\frac{|e^{-t^{2}}x-y|^{2}}{1-e^{-2t^{2}}})\exp(C\frac{t^{4}|x|^{2}}{1-e^{-2\frac{t^{2}}{\alpha}}}).$
\item
$\exp(-C\frac{|e^{-\frac{t^{2}}{\alpha}}x-y|^{2}}{1-e^{-2\frac{t^{2}}{\alpha}}})
\leq \exp(-C\frac{\alpha}{2e^{2a^{2}}}\frac{|e^{-t^{2}}x-y|^{2}}{1-e^{-2t^{2}}})\exp(C\frac{t^{4}|y|^{2}}{1-e^{-2\frac{t^{2}}{\alpha}}}).$

\end{enumerate}
 \end{lemma}

\begin{proof}
Let $t\in (0,a]$ and $\alpha>1 $. Applying the mean value theorem to $f(\xi) = \xi^\alpha$, we have
 \begin{align*}
    \frac{1-e^{-2t^2}}{1-e^{-\frac{2t^2}{\alpha}}}
       = \alpha \hat\xi^{\alpha-1}
 \end{align*}
for some $\hat \xi \in [e^{-2t^2/\alpha} ,1].$ Therefore,
 \begin{align*}
    \alpha e^{-2a^{2}}
     \leq \alpha e^{-\frac{2t^2(\alpha-1)}{\alpha}}
     \leq \frac{1-e^{-2t^2}}{1-e^{-\frac{2t^2}{\alpha}}}
     \leq \alpha.
 \end{align*}
To prove (i), we notice that 
$$
|e^{-\frac{t^{2}}{\alpha}}x-y| \geq |e^{-t^{2}}x-y| - |e^{-t^{2}}-e^{-\frac{t^{2}}{\alpha}}||x|
\geq |e^{-t^{2}}x-y| -t^{2}|x|,
$$
and thus, by Cauchy-Schwarz:
$$
|e^{-\frac{t^{2}}{\alpha}}x-y|^{2} \geq \frac{|e^{-t^{2}}x-y|^{2}}{2} - t^{4}|x|^{2}.
$$
This gives
\begin{equation*}
\begin{split}
\exp(-C\frac{|e^{-\frac{t^{2}}{\alpha}}x-y|^{2}}{1-e^{-2\frac{t^{2}}{\alpha}}})
&\leq \exp(-\frac{C}{2}(\frac{1-e^{-2t^{2}}}{1-e^{-2\frac{t^{2}}{\alpha}}})\frac{|e^{-t^{2}}x-y|^{2}}{1-e^{-2t^{2}}})\exp(C\frac{t^{4}|x|^{2}}{1-e^{-2\frac{t^{2}}{\alpha}}})
\\&\leq \exp(-C\frac{\alpha}{2e^{2a^{2}}}\frac{|e^{-t^{2}}x-y|^{2}}{1-e^{-2t^{2}}})\exp(C\frac{t^{4}|x|^{2}}{1-e^{-2\frac{t^{2}}{\alpha}}}).
\end{split}
\end{equation*}
The estimate (ii) is proven in the same way, noticing that
$$
|e^{-\frac{t^{2}}{\alpha}}x-y| \geq e^{(\frac{\alpha-1}{\alpha})t^{2}}|e^{-t^{2}}x-e^{-(\frac{\alpha-1}{\alpha})t^{2}}y|
\geq  |e^{-t^{2}}x-y|- |1-e^{-(\frac{\alpha-1}{\alpha})t^{2}}||y|
\geq |e^{-t^{2}}x-y| -t^{2}|y|.
$$
 \end{proof}

 \begin{lemma}
 \label{lem:est}
Let $N \in \Z_{+}$, $j \in \{1,...,n\}$, $a>0$ and $\alpha \geq 4e^{2a^{2}}$. Let $x,y \in \R^{n}$ and $t\in (0,a]$.  
\begin{enumerate}[(i)]
\item  If $t \lesssim m(y)$ then $M_{\frac{t^{2}}{\alpha}}(x,y) \lesssim \exp(-\frac{\alpha}{2e^{2a^{2}}}\frac{|e^{-t^{2}}x-y|^{2}}{1-e^{-2t^{2}}})M_{t^{2}}(x,y)$.
 \item If $t \lesssim m(x)$ then $|K_{t^{2},N,\alpha}(x,y)|\lesssim  
 \exp(-\frac{\alpha}{4e^{2a^{2}}}\frac{|e^{-t^{2}}x-y|^{2}}{1-e^{-2t^{2}}})M_{t^{2}}(x,y)$.
 \item  If $t \lesssim m(y)$ then $|\tilde{K}_{t^{2},N,\alpha,j}(x,y)|\lesssim \exp(-\frac{\alpha}{4e^{2a^{2}}}\frac{|e^{-t^{2}}y-x|^{2}}{1-e^{-2t^{2}}})M_{t^{2}}(x,y)$.
 \end{enumerate}
 \end{lemma}
\begin{proof}
(i) follows from Lemma \ref{lem:slow2}.\\
(ii) follows from Lemma \ref{lem:K} and Lemma \ref{lem:slow2} using that $\underset{w>0}{\sup}w^{k}e^{-Cw^{2}}<\infty$ for all $k\geq0$ and $C>0$.\\
(iii) follows from Corollary \ref{cor:Ktilde} and Lemma \ref{lem:slow2} in the same way, using that $$M_{t^{2}}(y,x)\exp(|x|^{2}-|y|^{2}) = M_{t^{2}}(x,y).$$
\end{proof}
We can now prove our main lemma.
\begin{lemma}[Off-diagonal estimates]
\label{lem:od}
Let $N \in \Z_{+}$, $a>0$, $j \in \{1,...,n\}$, $B \in \mathcal{B}_{a}$, $\alpha>4e^{2a^{2}}$, and $k \in \N$. Then for all $u \in L^{2}(\gamma)$
$$
\|1_{C_{k}(B)}1_{(0,r_{B})}(t)(t^{2N+1}L^{N}e^{\frac{t^{2}}{\alpha}L}\partial_{x_{j}}^{*})1_{B}u\|_{2} \lesssim
\exp(-\frac{\alpha}{2^{6}e^{2a^{2}}}4^{k}(\frac{r_{B}}{t})^{2})\|u\|_{2},
$$
with implied constant depending only on $\alpha$, $a$ and $N$.
\end{lemma}
\begin{proof}
For $t\leq r_{B} \leq am(c_{B})$ and $y \in B$, we have $t \leq a(1+a)m(y)$ by Lemma \ref{lem:mnp1}. 
Given $x\in \R^{n}$, we also have, using Cauchy-Schwarz, $|y-x|^{2} \leq 2(|e^{-t^{2}}y-x|^{2}+(1-e^{-t^{2}})^{2}|y|^{2})$, and thus
\begin{equation*}
\begin{split}
\exp(-\frac{\alpha}{2^{3}e^{2a^{2}}}\frac{|e^{-t^{2}}y-x|^{2}}{t^{2}})&\leq
\exp(-\frac{\alpha}{2^{4}e^{2a^{2}}}\frac{|y-x|^{2}}{t^{2}})\exp(\frac{\alpha}{2^{3}e^{2a^{2}}}(t|y|)^{2}) \\&\lesssim \exp(-\frac{\alpha}{2^{4}e^{2a^{2}}}\frac{|y-x|^{2}}{t^{2}}).
\end{split}
\end{equation*}
Therefore, using Lemma \ref{lem:est}, we have the following estimates.
\begin{equation*}
\begin{split}
\int \limits _{C_{k}(B)} &\left( \int \limits _{B} |\tilde{K}_{t^{2},N,\alpha,j}(x,y)|1_{(0,r_{B})}(t)|u(y)|dy \right)^{2} d\gamma(x)\\
& \lesssim \int \limits _{C_{k}(B)} \left( \int \limits _{B} \exp(-\frac{\alpha}{2^{3}e^{2a^{2}}}\frac{|e^{-t^{2}}y-x|^{2}}{t^{2}})M_{t^{2}}(x,y)1_{(0,r_{B})}(t)|u(y)|dy \right)^{2} d\gamma(x) \\
& \lesssim \exp(-\frac{\alpha}{2^{6}e^{2a^{2}}}4^{k}(\frac{r_{B}}{t})^{2})\|e^{t^{2}L}|u|\|_{2}
\lesssim \exp(-\frac{\alpha}{2^{6}e^{2a^{2}}}4^{k}(\frac{r_{B}}{t})^{2})\|u\|_{2} ^{2}.
\end{split}
\end{equation*}
\end{proof}

We conclude this section with a property of the sets $C_{k}(B)$ in the local region
$N_{\tau}(B):= \{x \in \R^{n} \;;\; |x-c_{B}| \leq \tau m(c_{B})\}$, which will be helpful when off-diagonal estimates fail.

\begin{lemma}\label{lem:region}
Let $a,\tau>0$ and $B=B(c_{B},r_{B}) \in \mathcal{B}_{a}$. There exists $C>0$ such that for all $k \in \Z_{+}$
$$
\gamma(C_{k}(B)\cap N_{\tau}(B)) \leq C 2^{kn}\gamma(B).
$$
\end{lemma}
\begin{proof}
Let $k \in \Z_{+}$ and $x\in C_{k}(B) \cap N_{\tau}(B)$.
We have $|x-c_{B}| \leq \tau m(c_{B}) \leq \tau(1+\tau)m(x)$, by Lemma \ref{lem:mnp1}.
Therefore
\begin{equation*}
\begin{split}
|x|^{2} \geq |c_{B}|^{2}-2\tau m(c_{B})|c_{B}| \\
|c_{B}|^{2} \geq |x|^{2}-2\tau(1+\tau) m(x)|x|,
\end{split}
\end{equation*}
and thus $e^{-|x|^{2}}\sim e^{-|c_{B}|^{2}}$ for all $x\in C_{k}(B) \cap N_{\tau}(B)$, with implicit constants independent of $k,B$ and $x$.
In particular, for $k=0$, we have 
$$
\gamma(B) \sim e^{-|c_{B}|^{2}}\int \limits _{B}dx \sim r_{B}^{n}e^{-|c_{B}|^{2}}.
$$
For $k \in \Z_{+}$, this gives
$$
\gamma(C_{k}(B)\cap N_{\tau}(B)) \lesssim \int \limits _{2^{k+1}B}e^{-|c_{B}|^{2}}dx
\lesssim (2^{k}r_{B})^{n}e^{-|c_{B}|^{2}} \lesssim 2^{kn} \gamma(B).
$$
\end{proof}

\section{Molecules}
In this section,  we show that, given a $t^{1,2}(\gamma)$ atom $F$ associated with a ball $B=B(c_{B},r_{B})\in \mathcal{B}_{2}$, the function
$$
\int \limits _{0} ^{r_{B}} (t^{2}L)^{N}e^{\frac{t^{2}}{\alpha}L}t\partial_{x_{j}}^{*}F(t,.)\frac{dt}{t}
$$
is a $(2,N,2^{-23}\alpha)$-molecule in the following sense.
\begin{definition}
Let $N \in \N$, $a>0$, and $C>0$. A function $f \in L^{2}(\gamma)$ is called a $(a,N,C)$-molecule if there exist $B=B(c_{B},r_{B})\in \mathcal{B}_{a}$
and $\tilde{f}$ in $L^{2}(\gamma)$ such that the following holds:
\begin{enumerate}[(i)]
\item
$\|1_{C_{k}(B)}f\|_{2} \leq e^{-C4^{k}}\gamma(B)^{-\frac{1}{2}} \quad \forall k \in \Z_{+}$,
\item
$f = L^{N}\tilde{f}$,
\item
$\|1_{C_{k}(B)}\tilde{f}\|_{2} \leq r_{B}^{2N}e^{-C4^{k}}\gamma(B)^{-\frac{1}{2}} \quad \forall k \in \Z_{+}$.
\end{enumerate}
\end{definition}
We then show that there exists $M>0$ depending only on $(a,N,C)$, such that $\|f\|_{h^{1}_{\text{max}}}\leq M$ for all $(a,N,C)$-molecules.
\begin{proposition}
\label{prop:mol}
Let $N \in \N$, $j\in \{1,...,n\}$ and $\alpha>0$. Let $B=B(c_{B},r_{B})\in \mathcal{B}_{2}$ and $F$ be a $t^{1,2}(\gamma)$ atom $F$ associated with $B$. The function
$$
\int \limits _{0} ^{r_{B}} (t^{2}L)^{N}e^{\frac{t^{2}}{\alpha}L}t\partial_{x_{j}}^{*}F(t,.)\frac{dt}{t}
$$
is a $(2,N,2^{-23}\alpha)$-molecule.
\end{proposition}
\begin{proof}
Let us treat the case $k=0$ first. Let $g = \sum \limits _{\beta \in \Z^{n}_{+}} c_{\beta}H_{\beta} \in L^{2}(\R^{n},\gamma)$ be such that
$ \sum \limits _{\beta \in \Z^{n}_{+}} |c_{\beta}|^{2} \leq 1$. We need to estimate
$$
\int \limits _{0} ^{r_{B}} \int \limits _{\R^{n}} |(t^{2}L)^{N}e^{\frac{t^{2}}{\alpha}L}t\partial_{x_{j}}^{*}F(t,x)g(x)| d\gamma(x)\frac{dt}{t}.
$$
By duality, and the $L^2$ boundedness of the Riesz transforms, we have that
\begin{equation*}
\begin{split}
\int \limits _{0} ^{r_{B}}  \int \limits _{\R^{n}} |(t^{2}L)^{N}e^{\frac{t^{2}}{\alpha}L}t\partial_{x_{j}}^{*}F(t,x)g(x)| d\gamma(x)\frac{dt}{t}
& \lesssim (\int \limits _{0} ^{r_{B}} \int \limits _{\R^{n}} |F(t,x)|^{2} d\gamma(x)\frac{dt}{t})^{\frac{1}{2}}
(\int \limits _{0} ^{r_{B}} \sum \limits _{\beta \in \Z^{n}_{+}} |(t^{2}|\beta|)^{N+\frac{1}{2}}e^{-\frac{t^{2}}{\alpha}|\beta|}c_{\beta}|^{2} \frac{dt}{t})^{\frac{1}{2}} \\
&\lesssim \gamma(B)^{-\frac{1}{2}} (\sum \limits _{\beta \in \Z^{n}_{+}} |c_{\beta}|^{2} \int \limits _{0} ^{\infty}
(t^{2}|\beta|)^{2N+1}e^{-\frac{2t^{2}}{\alpha}|\beta|} \frac{dt}{t})^{\frac{1}{2}} \lesssim \gamma(B)^{-\frac{1}{2}}.
\end{split}
\end{equation*}
Moreover $\int \limits _{0} ^{r_{B}} (t^{2}L)^{N}e^{\frac{t^{2}}{\alpha}L}t\partial_{x_{j}}^{*}F(t,.)\frac{dt}{t} = L^{N}\tilde{f}$ for
$\tilde{f}:=\int \limits _{0} ^{r_{B}} t^{2N+1}e^{\frac{t^{2}}{\alpha}L}\partial_{x_{j}}^{*}F(t,.)\frac{dt}{t}$.
The same argument thus gives
$$
\|\tilde{f}\|_{2} \lesssim r_{B}^{2N}\gamma(B)^{-\frac{1}{2}} (\int \limits _{0} ^{\infty}
t^{2}|\beta|e^{-\frac{2t^{2}}{\alpha}|\beta|} \frac{dt}{t})^{\frac{1}{2}} \lesssim r_{B}^{2N}\gamma(B)^{-\frac{1}{2}}.
$$
Now let $k \in \Z_{+}$ be such that $k \neq 0$. By Lemma \ref{lem:od}, we have the following.
\begin{equation*}
\begin{split}
\|1_{C_{k}(B)}\int \limits _{0} ^{r_{B}} (t^{2}L)^{N}e^{\frac{t^{2}}{\alpha}L}t\partial_{x_{j}}^{*}F(t,.)\frac{dt}{t}\|_{2}
&\lesssim \int \limits _{0} ^{r_{B}} \exp(-\frac{\alpha}{2^{6}e^{8}}4^{k}(\frac{r_{B}}{t})^{2})\|F(t,.)\|_{2} \frac{dt}{t}\\
&\lesssim  \exp(-\frac{\alpha}{2^{23}}4^{k}) 
(\int \limits _{0} ^{1} \exp(-\frac{\alpha}{2^{22}}(\frac{1}{t})^{2})\frac{dt}{t})^{\frac{1}{2}} (\int \limits _{0} ^{r_{B}} \|F(t,.)\|_{2} ^{2} \frac{dt}{t})^{\frac{1}{2}}\\
& \lesssim \exp(-\frac{\alpha}{2^{23}}4^{k}) \gamma(B)^{-\frac{1}{2}}.
\end{split}
\end{equation*}
Since
$$
\|1_{C_{k}(B)}\tilde{f}\|_{2} \leq r_{B}^{2N} \int \limits _{0} ^{r_{B}} \|1_{C_{k}(B)}1_{(0,r_{B})}(t)e^{\frac{t^{2}}{\alpha}L}t\partial_{x_{j}}^{*}F(t,.)\|_{2}
\frac{dt}{t},
$$
the proof is concluded as above, using Lemma \ref{lem:od} with $N$ replaced by $0$.
\end{proof}
\begin{theorem}
\label{thm:mol}
Let $a>0$, and $f$ a $(2,N,C)$-molecule with $N>\frac{n}{4}$ and $C>2^{11}$.
Then $f \in h^{1}_{\text{max},a}$ and $\|f\|_{h^{1}_{\text{max},a}} \leq M$ for some $M$ independent of $f$.
\end{theorem}
\begin{proof}
Let $B=B(c_{B},r_{B}) \in \mathcal{B}_{2}$ be the ball associated with $f$.
Pick $\alpha>2^{31}$, and let $C_{a}:=(4+4a)\tau+2a$ where $\tau:=\frac{(1+a)(1+2a)}{2}$ as in Proposition \ref{prop:glob}.
We use the following decomposition:
$$
\|f\|_{h^{1}_{\text{max},a}} \leq I + \sum \limits _{k=0} ^{\infty} \sum \limits _{l=0} ^{\infty} I'_{k,l} + \sum \limits _{k=0} ^{\infty} \sum \limits _{l=0} ^{\infty} I''_{k,l},  $$
where
\begin{equation*}
\begin{split}
I&:= \int \limits _{\R^{n}} \sup\{ |e^{s^{2}L}f(y)| \;;\; (y,s) \in \Gamma_{x} ^{a}(\gamma), 
s\leq \frac{r_{B}}{\sqrt{\alpha}}\}d\gamma(x),\\
I'_{k,l}&:= \int \limits _{C_{k}(B)} \sup\{  |e^{s^{2}L}(1_{C_{l}(B)}f)(y)| \;;\; (y,s) \in \Gamma_{x} ^{a}(\gamma), 
s\geq \frac{r_{B}}{\sqrt{\alpha}}\}1_{(0,\frac{2^{k}r_{B}}{C_{a}})}(m(x))d\gamma(x),\\
I''_{k,l}&:= \int \limits _{C_{k}(B)} \sup\{ |L^{N}e^{s^{2}L}(1_{C_{l}(B)}\tilde{f})(y)| \;;\; (y,s) \in \Gamma_{x} ^{a}(\gamma), 
s\geq \frac{r_{B}}{\sqrt{\alpha}}\}1_{[\frac{2^{k}r_{B}}{C_{a}},1]}(m(x))d\gamma(x).
\end{split}
\end{equation*}
{\em Estimating $I$:}
Decomposing into a local and global part and using Proposition \ref{prop:glob}, we have that
$$
I \lesssim \|f\|_{1} + \sum \limits _{k=0} ^{\infty} \sum \limits _{l=0} ^{\infty} I^{loc}_{k,l},
$$
where
$$
I ^{loc} _{k,l}:= \int \limits _{C_{k}(B)} \sup\{ \int \limits _{C_{l}(B)} M_{s^{2}}(z,w)1_{N_{\tau}}(z,w)|f(w)|dw \;;\; (z,s) \in \Gamma_{x} ^{a}(\gamma), s\leq \frac{r_{B}}{\sqrt{\alpha}}\}d\gamma(x).
$$
By Lemma \ref{lem:mm} we also have that
$$
\|f\|_{1} \leq \sum \limits _{k=0} ^{\infty} \sqrt{\gamma(2^{k+1}B)} \|1_{C_{k}(B)}f\|_{2}
\leq \sum \limits _{k=0} ^{\infty} e^{8(2^{k+2}+1)^{2}}e^{-C4^{k}} \lesssim 1,
$$
since $C>2^{9}$.\\
{\em Estimating $I_{k,l}^{loc}$ for $k< l+2$:}\\
By Lemma \ref{lem:mm} and Proposition \ref{prop:glob} we have that
$$
I_{k,l}^{loc} \leq \sqrt{\gamma(2^{k+1}B)} \|x\mapsto \sup\{e^{s^{2}L}|1_{C_{l}(B)}f|(y) \;;\; (y,s) \in \Gamma_{x}^{a}\}\|_{2}
\lesssim e^{2^{9}.4^{k}}\sqrt{\gamma(B)}\|1_{C_{l}(B)}f\|_{2} \leq e^{2^{9}.4^{k}} e^{-C.4^{l}},
$$
and thus:
$$
\sum \limits _{l=0} ^{\infty} \sum \limits _{k=0} ^{l+1} I^{loc}_{k,l} \leq \sum \limits _{l=0} ^{\infty} (l+2)e^{-(C-2^{11})4^{l}} \lesssim 1.
$$
{\em Estimating $I_{k,l}^{loc}$ for $k\geq l+2$:}\\
We use Lemma \ref{lem:est} as follows:
\begin{equation*}
\begin{split}
I_{k,l} ^{loc} &=  \int \limits _{C_{k}(B)} \sup\{ \int \limits _{C_{l}(B)} M_{\frac{t^{2}}{\alpha}}(z,w)1_{N_{\tau}}(z,w)|f(w)|dw \;;\; (z,t) \in 
\Gamma_{x} ^{(\frac{1}{\sqrt{\alpha}},a\sqrt{\alpha})}(\gamma), t\leq r_{B}\}d\gamma(x)\\
 &\lesssim  \int \limits _{C_{k}(B)} \sup\{ \int \limits _{C_{l}(B)} M_{t^{2}}(z,w)\exp(-\frac{\alpha}{2^{17}}\frac{|e^{-t^{2}}z-w|^{2}}{1-e^{-2t^{2}}})1_{N_{\tau}}(z,w)|f(w)|dw \;;\; (z,t) \in 
\Gamma_{x} ^{(\frac{1}{\sqrt{\alpha}},a\sqrt{\alpha})}(\gamma), t\leq r_{B}\}d\gamma(x),
\end{split}
\end{equation*}
where we have used Lemma \ref{lem:mnp1} to see that
\begin{equation*}
\begin{split}
|z-x| \leq am(x) &\implies m(x) \leq (1+a)m(z),\\
|z-w| \leq \tau m(z) &\implies m(z) \leq (1+\tau)m(w),\\
t \leq a\sqrt{\alpha}m(x) &\implies t \leq a\sqrt{\alpha}(1+a)(1+\tau)m(w).
\end{split}
\end{equation*}
Now, for $x \in C_{k}(B)$, $w \in C_{l}(B)$, $t \leq \min(r_{B},  a\sqrt{\alpha}(1+a)m(z))$, and $z \in B(x,\frac{t}{\sqrt{\alpha}})$, we have
$$
|e^{-t^{2}}z-w| \geq |x-w|-|x-z|-(1-e^{-t^{2}})|z|
\geq (2^{k-1}-\frac{2}{\sqrt{\alpha}}-2a\sqrt{\alpha}(1+a))r_{B}.
$$
Let $M_{a,\alpha} \in \N$ be such that $\frac{2}{\sqrt{\alpha}}+2a\sqrt{\alpha}(1+a)\leq 2^{M_{a,\alpha}}$.
Then, for $k \geq \max(l,M_{a,\alpha})+2$ we have the following.
\begin{equation*}
\begin{split}
I_{k,l}^{loc} &\lesssim \exp(-\frac{\alpha}{2^{18}}(2^{k-2})^{2})
\int \limits _{C_{k}(B)} \sup\{ e^{t^{2}L}|1_{C_{l}(B)}f|(z)\;;\; (z,t) \in 
\Gamma_{x} ^{(\frac{1}{\sqrt{\alpha}},\sqrt{\alpha}a)}(\gamma)\}d\gamma(x)\\
&\lesssim \exp(-\frac{\alpha}{2^{22}}4^{k}) \sqrt{\gamma(2^{k+1}B)}\|1_{C_{l}(B)}f\|_{2}
\leq \exp(-\frac{\alpha}{2^{22}}4^{k})\exp(2^{9}.4^{k})\exp(-C4^{l}),
\end{split}
\end{equation*}
where we have used Proposition \ref{prop:glob} and Lemma \ref{lem:mm}. Noticing that
$$
\sum \limits _{k=0} ^{M_{a,\alpha}+2} \sum \limits _{l=0} ^{M_{a,\alpha}} I_{k,l} ^{loc}
\lesssim \sum \limits _{k=0} ^{M_{a,\alpha}+2} \sum \limits _{l=0} ^{M_{a,\alpha}} \sqrt{\gamma(2^{k+1}B)}\|f\|_{2}
\leq \sum \limits _{k=0} ^{M_{a,\alpha}+2} \sum \limits _{l=0} ^{M_{a,\alpha}} \exp(2^{9}.4^{k}) \lesssim 1,
$$
and using the fact that $\alpha>2^{31}$, 
we get that $\sum \limits _{l=0} ^{\infty}\sum \limits _{k=l+2} ^{\infty} I_{k,l}^{loc} \lesssim 1$ and thus that $I \lesssim 1$.\\
{\em Estimating $I'_{k,l}$ for $k < l+2$:}\\
Reasoning as above, using Proposition \ref{prop:glob} and Lemma \ref{lem:mm}, we have that
$$
I'_{k,l} \lesssim \exp(2^{9}.4^{k}) \sqrt{\gamma(B)} \|1_{C_{l}(B)}f\|_{2} \lesssim \exp(2^{9}.4^{k}-C4^{l}),
$$
and thus $$
\sum \limits _{l=0} ^{\infty} \sum \limits _{k=0} ^{l+1} I'_{k,l} \leq \sum \limits _{l=0} ^{\infty} (l+2)\exp(-(C-2^{13})4^{l}) \lesssim 1.
$$
{\em Estimating $I'_{k,l}$ for $k \geq l+2$:}\\
Given $x \in C_{k}(B)$ such that $m(x) \leq \frac{2^{k}r_{B}}{C_{a}}$, $s\leq am(x)$,
 $y \in B(x,s)$, and $w \in C_{l}(B)$, we have, using Lemma \ref{lem:mnp1}:
$$
|y-w| \geq |x-w|-|x-y| \geq 2^{k-1}r_{B}(2-2^{l+2-k})-am(x) \geq (\frac{C_{a}}{2}-a)m(x) \geq \frac{1}{2+2a}(\frac{C_{a}}{2}-a)m(y)= \tau m(y).
$$
By Proposition \ref{prop:glob}, we thus have
$$
\sum \limits _{l=0} ^{\infty} \sum \limits _{k=l+2} ^{\infty} I'_{k,l}\leq \sum \limits _{l=0} ^{\infty} \|T^{*}_{glob,a,1}|1_{C_{l}(B)}f|\|_{1}
\lesssim \|f\|_{1} \lesssim 1.
$$
{\em Estimating $I''_{k,l}$:}\\
For $x \in \R^{n}$, $t\leq a\sqrt{\alpha}m(x)$, $y \in B(x,\frac{t}{\sqrt{\alpha}})$,
we have $t \lesssim m(y)$ by Lemma \ref{lem:mnp1} and thus 
$$
|L^{N}e^{\frac{t^{2}}{\alpha}L}(1_{C_{l}(B)}\tilde{f})(y)| \lesssim t^{-2N} \int \limits _{C_{l}(B)}
|K_{t^{2},N,\alpha}(y,w)||\tilde{f}(w)|dw \lesssim t^{-2N} \int \limits _{C_{l}(B)} M_{t^{2}}(y,w)|\tilde{f}(w)|dw,
$$
by Lemma \ref{lem:est}.
Therefore
$$
I''_{k,l} \lesssim \int \limits _{C_{k}(B)} \sup \{t^{-2N}e^{t^{2}L}|1_{C_{l}(B)}\tilde{f}|(z) \;;\; (z,t) \in 
\Gamma_{x}^{(\frac{1}{\sqrt{\alpha}},a\sqrt{\alpha})}(\gamma), t \geq r_{B}\}1_{[\frac{2^{k}r_{B}}{C_{a}},1]}(m(x))d\gamma(x)
\lesssim r_{B}^{-2N}J_{k,l}^{glob}+J_{k,l}^{loc},
$$
where
\begin{equation*}
\begin{split}
J_{k,l}^{glob} &:= \int \limits _{C_{k}(B)} \sup \{\int \limits _{C_{l}(B)} M_{t^{2}}(z,w)1_{N_{\tau} ^{c}}(z,w) |\tilde{f}|(w)dw \;;\; (z,t) \in 
\Gamma_{x}^{(\frac{1}{\sqrt{\alpha}},a\sqrt{\alpha})}(\gamma)\}d\gamma(x),\\
J_{k,l}^{loc} &:= \int \limits _{C_{k}(B)} \sup \{t^{-2N}\int \limits _{C_{l}(B)} M_{t^{2}}(z,w)1_{N_{\tau}}(z,w) |\tilde{f}|(w)dw \;;\; (z,t) \in 
\Gamma_{x}^{(\frac{1}{\sqrt{\alpha}},a\sqrt{\alpha})}(\gamma), t \geq r_{B}\}1_{[\frac{2^{k}r_{B}}{C_{a}},1]}(m(x))d\gamma(x),
\end{split}
\end{equation*}
and $\tau$ is defined as in Proposition \ref{prop:glob} for the parameters $(\frac{1}{\sqrt{\alpha}},a\sqrt{\alpha})$.
Proposition \ref{prop:glob} then gives that
$$
\sum \limits _{l=0} ^{\infty} \sum \limits _{k=0} ^{\infty}J_{k,l} ^{glob} \lesssim \sum \limits _{l=0} ^{\infty} \|1_{C_{l}(B)}\tilde{f}\|_{1} \lesssim r_{B}^{2N}.
$$
For $x \in C_{k}(B)$ and $m(x) \geq \frac{2^{k}r_{B}}{C_{a}}$ we have
$$
|x-c_{B}|\leq 2^{k+1}r_{B} \leq 2C_{a}m(x) \leq 2C_{a}(1+2C_{a})m(c_{B})=:\tau'm(c_{B}).
$$
Therefore
$$
J_{k,l}^{loc} \leq \int \limits _{C_{k}(B)\cap N_{\tau'}(B)} \sup \{t^{-2N}\int \limits _{C_{l}(B)} M_{t^{2}}(z,w)1_{N_{\tau}}(z,w) |\tilde{f}|(w)dw \;;\; (z,t) \in 
\Gamma_{x}^{(\frac{1}{\sqrt{\alpha}},a\sqrt{\alpha})}(\gamma), t \geq r_{B}\}d\gamma(x).
$$
{\em Estimating $J_{k,l}^{loc}$ for $k < l+2$:}\\
Using Proposition \ref{prop:glob} and  Lemma \ref{lem:region}, we have
$$
\sum \limits _{l=0} ^{\infty} \sum \limits _{k=0} ^{l+1} J_{k,l}^{loc}
\lesssim r_{B}^{-2N} \sum \limits _{l=0} ^{\infty} \sum \limits _{k=0} ^{l+1} \sqrt{\gamma(C_{k}(B)\cap N_{\tau'}(B))}\|1_{C_{l}(B)}\tilde{f}\|_{2}
\lesssim \sum \limits _{l=0} ^{\infty} \exp(-C4^{l})\sum \limits _{k=0} ^{l+1} 2^{k\frac{n}{2}} \lesssim 1.
$$
{\em Estimating $J_{k,l}^{loc}$ for $k \geq l+2$:}\\
For $x \in \R^{n}$, $s\leq a\alpha m(x)$, $z \in B(x,am(x))$, and $(z,w) \in N_{\tau}$, we have $m(w)\sim m(z) \sim m(x)$ and thus $s \lesssim m(w)$.
Therefore, using Lemma \ref{lem:est} we have
\begin{equation*}
\begin{split}
J_{k,l} ^{loc} &\lesssim \int \limits _{C_{k}(B)\cap N_{\tau'}(B)} \sup \{s^{-2N} \int \limits _{C_{l}(B)}
M_{\frac{s^{2}}{\alpha}}(z,w)1_{N_{\tau}}(z,w)|\tilde{f}(w)|dw \;;\; (z,s) \in \Gamma_{x}^{(\frac{1}{\alpha},a\alpha)}(\gamma), s\geq \sqrt{\alpha}r_{B}\}d\gamma(x)\\
&\lesssim \int \limits _{C_{k}(B)\cap N_{\tau'}(B)} \sup \{s^{-2N} \int \limits _{C_{l}(B)}
M_{s^{2}}(z,w) \exp(-\frac{\alpha}{2^{17}}\frac{|e^{-s^{2}}z-w|^{2}}{1-e^{-2s^{2}}})|\tilde{f}(w)|dw \;;\; (z,s) \in \Gamma_{x}^{(\frac{1}{\alpha},a\alpha)}(\gamma)\}d\gamma(x).
\end{split}
\end{equation*}
For $x \in C_{k}(B)$, $w \in C_{l}(B)$, $s \leq \alpha am(x)$, and $z \in B(x,\frac{1}{\alpha}s)$ we have
$$
|e^{-s^{2}}z-w| \geq |x-w| -|x-z| - (1-e^{-s^{2}})|z| \geq 2^{k-1}r_{B}-(\frac{1}{\alpha}+\alpha(a+2a^{2}))s.
$$
Therefore, there exists $C_{\alpha}>0$ such that
\begin{equation*}
\begin{split}
J_{k,l}^{loc} & \lesssim \int \limits _{C_{k}(B)\cap N_{\tau'}(B)} \sup \{s^{-2N} \exp(-C_{\alpha}4^{k}(\frac{r_{B}}{s})^{2}) 
\int \limits _{C_{l}(B)} M_{s^{2}}(z,w)|\tilde{f}(w)|dw \;;\; (z,s) \in \Gamma_{x}^{(\frac{1}{\alpha},a\alpha)}(\gamma)\}d\gamma(x)\\
& \lesssim (2^{k}r_{B})^{-2N}\int \limits _{C_{k}(B)\cap N_{\tau'}(B)} \sup \{
\int \limits _{C_{l}(B)} M_{s^{2}}(z,w)|\tilde{f}(w)|dw \;;\; (z,s) \in \Gamma_{x}^{(\frac{1}{\alpha},a\alpha)}(\gamma)\}d\gamma(x)\\
& \lesssim (2^{k}r_{B})^{-2N}\sqrt{\gamma(C_{k}(B)\cap N_{\tau'}(B))}\|1_{C_{l}(B)}\tilde{f}\|_{2}
\lesssim 4^{-kN}\exp(-C4^{l})2^{k\frac{n}{2}},
\end{split}
\end{equation*}
where we have used Proposition \ref{prop:glob} and Lemma \ref{lem:region}.
This gives
$$
\sum \limits _{l=0} ^{\infty} \sum \limits _{k=0} ^{l+2} J^{loc}_{k,l}
\lesssim \sum \limits _{l=0} ^{\infty} \sum \limits _{k=0} ^{l+2} 4^{-k(N-\frac{n}{4})}\exp(-C4^{l}) \lesssim 1,
$$
which concludes the proof.
\end{proof}
\section{Remainder terms}
In this section, we handle the remainder terms
\begin{enumerate} 
\item
$
\int \limits _{0} ^{2} 1_{[\frac{m(.)}{b},2]}(t)t^{2N+1}L^{N}e^{\frac{t^{2}}{\alpha}L}\partial_{x_{j}}^{*}F(t,.) \frac{dt}{t}$,
\item
$
\int \limits _{0} ^{\frac{m(.)}{b}} t^{2N+1}L^{N}e^{\frac{(1+a^{2})t^{2}}{\alpha}L}\partial_{x_{j}}^{*}(1_{D^{c}}(t,.)t\partial_{x_{j}}e^{\frac{a^{2}t^{2}}{\alpha}L})u \frac{dt}{t}$,
\item
$ \int \limits _{\frac{m(.)}{b}} ^{\infty} t^{2N+2}L^{N}e^{\frac{(1+a^{2})t^{2}}{\alpha}L}u \frac{dt}{t}$,
\end{enumerate}
where $u \in L^{1}(\gamma)$ and $F$ is a $t^{1,2}(\gamma)$ atom.
\begin{lemma}
Let $N \in \Z_{+}$, $j \in \{1,...,n\}$, $b>0$ and $\alpha>2^{32}$.
Let $F$ be a $t^{1,2}(\gamma)$ atom associated with the ball $B=B(c_{B},r_{B}) \in \mathcal{B}_{2}$.
Then
$$
\|\int \limits _{0} ^{r_{B}} 1_{[\frac{m(.)}{b},2]}(t)t^{2N+1}L^{N}e^{\frac{t^{2}}{\alpha}L}\partial_{x_{j}}^{*}F(t,.) \frac{dt}{t}\|_{L^{1}} \lesssim 1.
$$
\end{lemma}
\begin{proof}
By Lemma \ref{lem:mnp1}, we have $m(y)\sim m(c_{B})$ for $y \in B$. Therefore, by Lemma \ref{lem:est}, and reasoning as in Proposition \ref{prop:mol}, we have
\begin{equation*}
\begin{split}
\|\int \limits _{0} ^{r_{B}} & 1_{[\frac{m(.)}{b},2]}(t)t^{2N+1}L^{N}e^{\frac{t^{2}}{\alpha}L}\partial_{x_{j}}^{*}F(t,.) \frac{dt}{t}\|_{L^{1}} 
\lesssim \sum \limits _{k=0} ^{\infty}
\int \limits _{C_{k}(B)} \int \limits _{0} ^{r_{B}} \int \limits _{B} |\tilde{K}_{t^{2},N,\alpha,j}(x,y)||F(t,y)| dy \frac{dt}{t} d\gamma(x)\\
&\lesssim 1+\sum \limits _{k=1} ^{\infty} \int \limits _{0} ^{r_{B}} \exp(-\frac{\alpha}{2^{22}}4^{k}(\frac{r_{B}}{t})^{2}) \sqrt{\gamma(2^{k+1}B)}\|F(t,.)\|_{2} \frac{dt}{t}\\
&\lesssim  1+\sum \limits _{k=1} ^{\infty} \exp(2^{9}.4^{k})\sqrt{\gamma(B)}\exp(-\frac{\alpha}{2^{23}}4^{k})
(\int \limits _{0} ^{r_{B}} \exp(-\frac{\alpha}{2^{22}}4^{k}(\frac{r_{B}}{t})^{2})\frac{dt}{t})^{\frac{1}{2}}\gamma(B)^{-\frac{1}{2}}\\
&\lesssim  1+\sum \limits _{k=1} ^{\infty} \exp(-(\frac{\alpha}{2^{23}}-2^{9})4^{k}) \lesssim 1.
\end{split}
\end{equation*}
\end{proof}

Combined with Proposition \ref{prop:glob}, this gives
\begin{corollary}
Let $a,b>0$, $N \in \Z_{+}$, $\{j=1,...,n\}$, and $\alpha>2^{32}$.
Let $F$ be a $t^{1,2}(\gamma)$ atom associated with the ball $B=B(c_{B},r_{B}) \in \mathcal{B}_{2}$.
Then
$$
\|T^{*}_{\text{glob},a}(\int \limits _{0} ^{r_{B}} 1_{[\frac{m(.)}{b},2]}(t)t^{2N+1}L^{N}e^{\frac{t^{2}}{\alpha}L}\partial_{x_{j}}^{*}F(t,.) \frac{dt}{t})\|_{1} \lesssim 1.
$$
\end{corollary}
\begin{proposition}
\label{prop:r1}
Let $a>0$, $N \in \Z_{+}$, $\{j=1,...,n\}$, and $\alpha>2^{38}$.
Let $F$ be a $t^{1,2}(\gamma)$ atom associated with the ball $B=B(c_{B},r_{B}) \in \mathcal{B}_{2}$.
Then
$$
\|\int \limits _{0} ^{r_{B}} 1_{[\frac{m(.)}{b},2]}(t)t^{2N+1}L^{N}e^{\frac{t^{2}}{\alpha}L}\partial_{x_{j}}^{*}F(t,.) \frac{dt}{t}\|_{h^{1}_{\text{max},a}} \lesssim 1.
$$
\end{proposition}
\begin{proof}
Given the above Corollary, and $\tau$ as in Proposition \ref{prop:glob}, we only have to estimate
$$
I = \int \limits _{\R^{n}} \sup \{ \int \limits _{\R^{n}} M_{s^{2}}(y,z)1_{N_{\tau}}(y,z)
\int \limits _{0} ^{r_{B}} \int \limits _{\R^{n}}1_{[\frac{m(z)}{b},2]}(t)|\tilde{K}_{t^{2},N,\alpha,j}(z,w)||F(t,w)|dw \frac{dt}{t} dz \;;\; (y,s) \in \Gamma_{x}^{a}(\gamma)\} d\gamma(x).
$$
For $w \in B$ and $t\leq r_{B}$, we have $t \lesssim m(w)$ by Lemma \ref{lem:mnp1}. Therefore, by Lemma \ref{lem:est}
\begin{equation*}
\begin{split}
I &\lesssim \int \limits _{\R^{n}} \underset{(y,s) \in \Gamma_{x}^{a}(\gamma)}{\sup} \int \limits _{\R^{n}} M_{s^{2}}(y,z)1_{N_{\tau}}(y,z)
\int \limits _{0} ^{r_{B}} \int \limits _{\R^{n}}1_{[\frac{m(z)}{b},2]}(t) \exp(-\frac{\alpha}{2^{23}}\frac{|e^{-t^{2}}w-z|^{2}}{1-e^{-2t^{2}}})M_{t^{2}}(z,w)
|F(t,w)|dw \frac{dt}{t} dz  d\gamma(x)\\
&\lesssim I_{loc}+\sum \limits _{k=0} ^{\infty} I^{glob} _{k},
\end{split}
\end{equation*}
where
\begin{equation*}
\begin{split}
I^{glob}_{k}:=&\int \limits _{C_{k}(B)} \underset{(y,s) \in \Gamma_{x}^{a}(\gamma)}{\sup} \int \limits _{\R^{n}} M_{s^{2}}(y,z)1_{N_{\tau}}(y,z)
\int \limits _{0} ^{r_{B}} \int \limits _{\R^{n}} 1_{[\frac{m(z)}{b},2]}(t) e^{-\frac{\alpha}{2^{23}}\frac{|e^{-t^{2}}w-z|^{2}}{1-e^{-2t^{2}}}}1_{N_{1} ^{c}}(z,w)M_{t^{2}}(z,w)
|F(t,w)|dw \frac{dt}{t} dz  d\gamma(x),\\
I_{loc}:=&\int \limits _{\R^{n}} \underset{(y,s) \in \Gamma_{x}^{a}(\gamma)}{\sup} \int \limits _{\R^{n}} M_{s^{2}}(y,z)1_{N_{\tau}}(y,z)
\int \limits _{0} ^{r_{B}} \int \limits _{\R^{n}}1_{[\frac{m(z)}{b},2]}(t) e^{-\frac{\alpha}{2^{23}}\frac{|e^{-t^{2}}w-z|^{2}}{1-e^{-2t^{2}}}}1_{N_{1}}(z,w)M_{t^{2}}(z,w)
|F(t,w)|dw \frac{dt}{t} dz  d\gamma(x).
\end{split}
\end{equation*}
{\em Estimating $I^{glob}_{k}$:}\\
For $w \in B$, $x \in C_{k}(B)$, $y \in B(x,am(x))$, $z \in B(y,\tau m(y))$, $t \leq r_{B}$, and $m(z)\leq br_{B}$, 
Lemma \ref{lem:mnp1}, gives that $t \lesssim m(w)$, $|x-z| \leq (a+2\tau(1+a))m(x)$ and $m(x) \leq (1+a+2\tau(1+a))m(z) \leq b(1+a+2\tau(1+a))r_{B}$.
Therefore
$$
|e^{-t^{2}}w-z| \geq |w-x|-|x-z|-(1-e^{-t^{2}})|w| \geq 2^{k-1}r_{B}-C_{a,b}r_{B},
$$
for some $C_{a,b}>0$.
Let $M_{a,b} \in \N$ be such that $C_{a,b} \leq 2^{M_{a,b}}$. 
We first notice that, for $k \leq M_{a,b}+1$, $ x \in C_{k}(B)$, and $z \in B(x,(a+2\tau(1+a))m(x))$  Lemma \ref{lem:mnp1} gives 
$m(z) \sim m(x) \sim m(c_{B})$ with implicit constant depending only on $a$ and $b$. In particular $\frac{m(z)}{b} \geq \kappa_{a,b}m(c_{B})$ for some
$\kappa_{a,b}>0$.
Therefore
\begin{equation*}
\begin{split}
\sum \limits _{k=0} ^{M_{a,b}+1} I^{glob}_{k} &\lesssim
\sum \limits _{k=0} ^{M_{a,b}+1}  \sqrt{\gamma(2^{k+1}B)} \int \limits _{\kappa_{a,b}m(c_{B})} ^{2m(c_{B})} 
\|T^{*}_{a}(e^{t^{2}L}|F(t,.)|)\|_{2} \frac{dt}{t} \\
&\lesssim
\sum \limits _{k=0} ^{M_{a,b}+1}  \sqrt{\gamma(B)}\exp(2^{9}.4^{k})
(\int \limits _{\kappa_{a,b}m(c_{B})} ^{2m(c_{B})}  \frac{dt}{t})^{\frac{1}{2}}
(\int \limits _{0} ^{r_{B}}\|F(t,.)\|_{2}  ^{2}\frac{dt}{t})^{\frac{1}{2}} \\
&\lesssim \sum \limits _{k=0} ^{M_{a,b}+1} \exp(2^{9}.4^{k}) \lesssim 1.
\end{split}
\end{equation*}
For $k \geq M_{a,b}+2$ we estimate as follows, using Lemma \ref{lem:est},
\begin{equation*}
\begin{split}
\sum \limits _{k=M_{a,b}+2} ^{\infty} I^{glob}_{k} &\lesssim
\sum \limits _{k=M_{a,b}+2} ^{\infty}  \sqrt{\gamma(2^{k+1}B)} \int \limits _{0} ^{r_{B}} \exp(-\frac{\alpha}{2^{28}}4^{k}(\frac{r_{B}}{t})^{2})
\|T^{*}_{a}(e^{t^{2}L}|F(t,.)|)\|_{2} \frac{dt}{t} \\
&\lesssim
\sum \limits _{k=M_{a,b}+2} ^{\infty}  \sqrt{\gamma(B)}\exp(2^{9}.4^{k})\exp(-\frac{\alpha}{2^{29}}4^{k}) 
(\int \limits _{0} ^{r_{B}} \exp(-\frac{\alpha}{2^{28}}(\frac{2^{k}r_{B}}{t})^{2}) \frac{dt}{t})^{\frac{1}{2}}
(\int \limits _{0} ^{r_{B}}\|F(t,.)\|_{2}  ^{2}\frac{dt}{t})^{\frac{1}{2}} \\
&\lesssim \sum _{k=M_{a,b}+2} ^{\infty}\exp(2^{9}.4^{k})\exp(-\frac{\alpha}{2^{29}}4^{k}) \lesssim 1.
\end{split}
\end{equation*}
{\em Estimating $I_{loc}$:}\\
We have
$$
I_{loc} \lesssim \int \limits _{\R^{n}} \underset{(y,s) \in \Gamma_{x}^{a}(\gamma)}{\sup} \int \limits _{\R^{n}} M_{s^{2}}(y,z)1_{N_{\tau}}(y,z)
\int \limits _{0} ^{r_{B}} \int \limits _{\R^{n}}1_{[\frac{m(z)}{b},2]}(t) 1_{N_{1}}(z,w)m(z)^{-n}|F(t,w)|dw \frac{dt}{t} dz  d\gamma(x).
$$
For $w \in B$, $(z,w) \in N_{1}$, $(y,z) \in N_{\tau}$, and $(x,y) \in N_{a}$, we have that $m(x)\sim m(y) \sim m(z) \sim m(w) \sim m(c_{B})$.
Moreover $|x-c_{B}|\leq am(x)+\tau m(y)+m(z)+m(c_{B}) \lesssim m(c_{B})$, $|x-w| \lesssim m(w)$, and $e^{-|w|^{2}}\sim e^{-|x|^{2}}$.
Let $\kappa, \lambda$ be such that $\frac{m(z)}{b} \geq \kappa m(c_{B})$ and $|x-c_{B}| \leq \lambda m(c_{B})$.
Using the positivity of $(e^{tL})_{t>0}$, and the fact that $e^{L}1=1$, we have that
\begin{equation*}
I_{loc} \lesssim \int \limits _{\kappa m(c_{B})} ^{r_{B}} m(c_{B})^{-n}\int \limits _{B(c_{B},\lambda m(c_{B}))} \|F(t,.)\|_{1} dx \frac{dt}{t} 
\lesssim (\int \limits _{\kappa m(c_{B})} ^{2m(c_{B})} \frac{dt}{t})^{\frac{1}{2}}\sqrt{\gamma(B)}  (\int \limits _{0} ^{r_{B}} \|F(t,.)\|_{2} \frac{dt}{t})^{\frac{1}{2}}
\lesssim 1.
\end{equation*}
\end{proof}
\begin{proposition}
\label{prop:Dcomp}
Let $a,a'>0$, $N \in \Z_{+}$, $j \in \{1,...,n\}$ and $\alpha > \max(32e^{4},4\sqrt{a}e^{2a^{2}})$. Let $b \geq 
\max(2e,\sqrt{\frac{32e^{4}}{(\alpha-32e^{4})(1-e^{-2\frac{a^{2}}{\alpha}})}})$.
Then $$
\|\int \limits _{0} ^{\frac{m(.)}{b}} t^{2N+1}L^{N}e^{\frac{t^{2}}{\alpha}L}\partial_{x_{j}}^{*}(1_{D^{c}}(t,.)t\partial_{x_{j}}e^{\frac{a^{2}t^{2}}{\alpha}L})u \frac{dt}{t}\|_{h^{1}_{\text{max},a'}} \lesssim \|u\|_{L^{1}(\gamma)}. $$
\end{proposition}
\begin{proof}
We claim that 
$$
\|\int \limits _{0} ^{\frac{m(.)}{b}} t^{2N+1}L^{N}e^{\frac{t^{2}}{\alpha}L}\partial_{x_{j}}^{*}(1_{D^{c}}(t,.)t\partial_{x_{j}}e^{\frac{a^{2}t^{2}}{\alpha}L})u \frac{dt}{t}\|_{\infty} \lesssim \|u\|_{1}. $$
The result then follows from the fact that $e^{sL}1=1$ for all $s>0$ and the positivity of $e^{sL}$.
To prove the claim, fix $x \in \R^{n}$, and consider $t\geq 0$ and $y \in \R^{n}$ such that $m(y) \leq t \leq \frac{m(x)}{b}$.
Then $|y| \geq 1$ and $|y| \geq b|x|\geq 2e|x|$. Therefore $|e^{-t^{2}}y-x|\geq \frac{|y|}{2e}+\frac{|y|}{2e}-|x|\geq\frac{|y|}{2e}$
and $t^{-1} \leq |y|$. Using Corollary \ref{cor:Ktilde} and Lemma \ref{lem:est}, this gives, for some $M>0$
\begin{equation*}
\begin{split}
t^{-1}|\tilde{K}_{t^{2},N,\alpha,j}(x,y)| &\lesssim |y|^{M}\exp(-\frac{\alpha}{2e^{2}}\frac{|e^{-t^{2}}y-x|^{2}}{1-e^{-2t^{2}}})M_{t^{2}}(x,y)\\
&\lesssim  |y|^{M+n}\exp(-\frac{\alpha}{16e^{4}}|y|^{2}) \lesssim \exp(-\frac{\alpha}{32e^{4}}|y|^{2}).
\end{split}
\end{equation*}
Using Lemma \ref{lem:slow2}, and the fact that $t\mapsto \frac{t^{2}}{1-e^{-\frac{2a^{2}t^{2}}{\alpha}}}$ is increasing on $(0,1)$,  we then have
\begin{equation*}
\begin{split}
\|\int \limits _{0} ^{\frac{m(.)}{b}} &t^{2N+1}L^{N}e^{\frac{t^{2}}{\alpha}L}\partial_{x_{j}}^{*}(1_{D^{c}}(t,.)t\partial_{x_{j}}e^{\frac{a^{2}t^{2}}{\alpha}L})u \frac{dt}{t}\|_{\infty} 
\lesssim \int \limits _{0} ^{\frac{1}{b}} \int \limits _{\R^{n}} \int \limits _{\R^{n}} \frac{|e^{-\frac{a^{2}t^{2}}{\alpha}}y_{j}-z_{j}|}{\sqrt{1-e^{-\frac{2a^{2}t^{2}}{\alpha}}}}
M_{\frac{a^{2}t^{2}}{\alpha}}(y,z)\exp(-\frac{\alpha}{32e^{4}}|y|^{2})|u(z)|dzdydt\\
&\lesssim \int \limits _{0} ^{\frac{1}{b}} \int \limits _{\R^{n}} \int \limits _{\R^{n}} 
t^{-n}\exp(-\frac{\alpha}{4\sqrt{a}e^{2a^{2}}}\frac{|e^{-t^{2}}y-z|^{2}}{1-e^{-2t^{2}}})\exp(\frac{t^{2}}{1-e^{-\frac{2a^{2}t^{2}}{\alpha}}}\frac{1}{2b^{2}}|y|^{2})
\exp(-\frac{\alpha}{32e^{4}}|y|^{2})|u(z)|dzdydt\\
&\lesssim \int \limits _{0} ^{\frac{1}{b}} \int \limits _{\R^{n}} \int \limits _{\R^{n}} 
M_{t^{2}}(y,z)\exp(\frac{1}{2b^{2}(1-e^{-\frac{2a^{2}}{\alpha}})}|y|^{2})
\exp(-\frac{\alpha}{32e^{4}}|y|^{2})|u(z)|dzdydt\\
&\lesssim \int \limits _{0} ^{\frac{1}{b}} \int \limits _{\R^{n}} e^{t^{2}L}|u(y)| d\gamma(y)dt \lesssim \|u\|_{1}.
\end{split}
\end{equation*}
\end{proof}

\begin{proposition}
\label{prop:jinf}
Let $N \in \Z_{+}$, $a,a',b>0$, and $\alpha>8e^{2a^{2}}$. For all $u\in C_{c}^{\infty}(\R^{n})$, we have
$$
\| \int\limits _{\frac{m(.)}{b}} ^{\infty} (t^{2}L)^{N+1}e^{\frac{(1+a^{2})t^{2}}{\alpha}L}u \frac{dt}{t}\|_{h^{1}_{\text{max},a'}} \lesssim \|u\|_{1}.
$$
\end{proposition}
\begin{proof}
Let $M>1$ and $x \in \R^{n}$. Without loss of generality we assume that $\int u d\gamma =0$ (since $L1=0$).
\begin{equation*}
\begin{split}
|\int \limits _{\frac{m(x)}{b}} ^{M} (t^{2}L)^{N+1}e^{(1+a^{2})\frac{t^{2}}{\alpha}L}u(x) \frac{dt}{t}|
&\lesssim |\int \limits _{\frac{(1+a^{2})m(x)^{2}}{b^{2}\alpha}} ^{\frac{(1+a^{2})M^{2}}{\alpha}} s^{N+1}\partial_{s} ^{N+1}e^{sL}u(x) \frac{ds}{s}|\\
&\lesssim \sum \limits _{k=0} ^{N} \int \limits _{\R^{n}} |K_{(1+a^{2})b^{-2}m(x)^{2},k,\alpha}(x,y)||u(y)|dy
+ \sum \limits _{k=0} ^{N} |(M^{2}L)^{k}e^{\frac{(1+a^{2})M^{2}}{\alpha}L}u(x)|.
\end{split}
\end{equation*}
Given $k \in \{0,...,N\}$ we have, using chaos decomposition and Proposition \ref{prop:glob}:
$$
\|(M^{2}L)^{k}e^{\frac{(1+a^{2})M^{2}}{\alpha}L}u\|_{h^{1}_{\text{max},a'}} \leq \|T^{*}_{a'}(M^{2}L)^{k}e^{\frac{(1+a^{2})M^{2}}{\alpha}L}u\|_{2}
\lesssim \|(M^{2}L)^{k}e^{\frac{(1+a^{2})M^{2}}{\alpha}L}u\|_{2} \leq M^{2k}e^{-\frac{(1+a^{2})M^{2}}{\alpha}}\|u\|_{2} \underset{M\to \infty}{\to} 0.
$$
It thus remains to prove that, given $k \in \{0,...,N\}$,
$$
\|T^{*}_{a'}(\int \limits _{\R^{n}} |K_{(1+a^{2})b^{-2}m(.)^{2},k,\alpha}(x,y)||u(y)|dy)\|_{1} \lesssim \|u\|_{1}.
$$
Using Lemma \ref{lem:est}, the positivity of $(e^{tL})_{t\geq 0}$, and the fact that $e^{L}1=1$, this further reduces to proving
$$
\|T^{*}_{a'}(\int \limits _{\R^{n}} M_{(1+a^{2})b^{-2}m(.)^{2}}(x,y)|u(y)|dy)\|_{1} \lesssim \|u\|_{1}.
$$
We first use Proposition \ref{prop:glob} to obtain
$$
\|T^{*}_{glob,a',1}(\int \limits _{\R^{n}} M_{(1+a^{2})b^{-2}m(.)^{2}}(x,y)|u(y)|dy)\|_{1} \lesssim \int \limits _{\R^{n}} \int \limits _{\R^{n}} M_{(1+a^{2})b^{-2}m(x)^{2}}(x,y)|u(y)|dy d\gamma(x).
$$
We decompose the right hand side into a local and a global part. Let $\tau:= \frac{1}{2}(1+b^{-1}\sqrt{1+a^{2}})(1+2b^{-1}\sqrt{1+a^{2}})$
and $\overline{\tau}=2(1+\sqrt{1+a^{2}}b^{-1})\tau+\sqrt{1+a^{2}}b^{-1}$. For $x,y,z \in \R^{n}$ such that $|x-y| \geq \overline{\tau} m(x)$ and
$|z-x| \leq \frac{\sqrt{1+a^{2}}}{b} m(x)$, we have that $|z-y|\geq \tau m(z)$. Therefore
$$
\int \limits _{\R^{n}} \int \limits _{\R^{n}} M_{(1+a^{2})b^{-2}m(x)^{2}}(x,y)1_{N_{\overline{\tau}}^{c}}(x,y)|u(y)|dy d\gamma(x)
\lesssim \int \limits _{\R^{n}} \underset{(z,t) \in \Gamma ^{b^{-1}\sqrt{1+a^{2}}} _{x}(\gamma)}{\sup}\int \limits _{\R^{n}} M_{t^{2}}(z,y)1_{N_{\tau}^{c}}(z,y)|u(y)|dy d\gamma(x)
\lesssim \|u\|_{1},
$$
by Proposition \ref{prop:glob}.
Now, for $(x,y) \in N_{\overline{\tau}}$, we have $m(x)\sim m(y)$ by Lemma \ref{lem:mnp1}. Therefore
$$
\int \limits _{\R^{n}} \int \limits _{\R^{n}} M_{(1+a^{2})b^{-2}m(x)^{2}}(x,y)1_{N_{\overline{\tau}}}(x,y)|u(y)|dy d\gamma(x)
\lesssim \int \limits _{\R^{n}} m(x)^{-n}\int \limits _{B(x,\overline{\tau} m(x))}|u(y)|dy d\gamma(x).
$$
For $(x,y) \in N_{\overline{\tau}}$, we also have $e^{-|x|^{2}}\sim e^{-|y|^{2}}$, therefore
$$
\int \limits _{\R^{n}} m(x)^{-n}\int \limits _{B(x,\overline{\tau} m(x))}|u(y)|dy d\gamma(x)
\lesssim \int \limits _{\R^{n}} |u(y)|m(y)^{-n}\int \limits _{B(y,\overline{\tau}(1+\overline{\tau})m(y))}d\gamma(x) dy
\lesssim \int \limits _{\R^{n}} |u(y)| e^{-|y|^{2}}dy \lesssim \|u\|_{1}.
$$
The proof will be completed once we have estimated the two following terms.
\begin{equation*}
\begin{split}
J_{glob}:=& \int \limits \underset{(y,t)\in \Gamma_{x} ^{a}}{\sup} \int \limits _{\R^{n}} M_{t^{2}}(y,z) 1_{N_{\tau'}}(y,z)
\int \limits _{\R^{n}} M_{(1+a^{2})b^{-2}m(z)^{2}}(z,w)1_{N_{\tau''}^{c}}(z,w)|u(w)|dwdzd\gamma(x),\\
J_{loc}:=& \int \limits \underset{(y,t)\in \Gamma_{x} ^{a}}{\sup} \int \limits _{\R^{n}} M_{t^{2}}(y,z) 1_{N_{\tau'}}(y,z)
\int \limits _{\R^{n}} M_{(1+a^{2})b^{-2}m(z)^{2}}(z,w)1_{N_{\tau''}}(z,w)|u(w)|dwdzd\gamma(x),
\end{split}
\end{equation*}
where $\tau'$ is defined in Proposition \ref{prop:glob} for the parameters $(1,a')$, and $\tau''$ is defined as follows.
For $(x,y) \in N_{a}$ and $(y,z) \in N_{\tau'}$, we have $m(x)\sim m(y) \sim m(z)$ by Lemma \ref{lem:mnp1}.
Let $\lambda>0$ be such that $\lambda^{-1} m(x) \leq m(z) \leq \lambda m(x)$, and fix 
$\tau''$ as in Proposition \ref{prop:glob}, for the parameters $(\tilde{A},\tilde{a}) = ((2\tau'(1+a)+a)b/(\lambda\sqrt{1+a^{2}}),\sqrt{1+a^{2}}b^{-1}\lambda)$. 
Using Proposition \ref{prop:glob}, the positivity of $(e^{tL})_{t\geq 0}$, and the fact that $e^{L}1=1$, we have that
\begin{equation*}
\begin{split}
J_{glob} &\lesssim \int \limits \underset{(y,t)\in \Gamma_{x} ^{a}}{\sup} \int \limits _{\R^{n}} M_{t^{2}}(y,z) 1_{N_{\tau'}}(y,z)
\underset{(\eta,s)\in \Gamma_{x} ^{(\tilde{A},\tilde{a})}(\gamma)}{\sup}\int \limits _{\R^{n}} M_{s^{2}}(\eta,w)1_{N_{\tau''}^{c}}(\eta,w)|u(w)|dwdzd\gamma(x)\\
&\lesssim \int \limits \underset{(\eta,s)\in \Gamma_{x} ^{(\tilde{A},\tilde{a})}(\gamma)}{\sup}\int \limits _{\R^{n}} M_{s^{2}}(\eta,w)1_{N_{\tau''}^{c}}(\eta,w)|u(w)|dwd\gamma(x)\\
&\lesssim \|u\|_{1}, 
\end{split}
\end{equation*}
Finally, for $(x,y) \in N_{a}$, $(y,z) \in N_{\tau'}$, and $(z,w) \in N_{\tau''}$, we have $m(x)\sim m(y) \sim m(z) \sim m(w)$,
 $|w-x| \leq \lambda m(x)$
for some numerical constant $\lambda>0$ by Lemma \ref{lem:mnp1}, and $e^{-|w|^{2}}\sim e^{-|x|^{2}}$. Let $C>0$ be such that $m(x) \leq Cm(w)$. 
Using  the positivity of $(e^{tL})_{t\geq 0}$, and the fact that $e^{L}1=1$, we have that
\begin{equation*}
\begin{split}
J_{loc}\lesssim& \int \limits \underset{(y,t)\in \Gamma_{x} ^{a}}{\sup} \int \limits _{\R^{n}} M_{t^{2}}(y,z) 1_{N_{\tau'}}(y,z)
m(x)^{-n}\int \limits _{B(x,\lambda m(x))} |u(w)|dwdzd\gamma(x) \\
&\lesssim \int \limits m(x)^{-n}\int \limits _{B(x,\lambda m(x))} |u(w)|dwd\gamma(x)
\lesssim  \int \limits  |u(w)| m(w)^{-n}\int \limits _{B(w,C\lambda m(w))}d\gamma(x)dw\\
&\lesssim \int \limits  |u(w)| e^{-|w|^{2}}dw \lesssim \|u\|_{1}.
\end{split}
\end{equation*}
\end{proof}

\section{Riesz transforms}
\label{sect:riesz}
In this section, we prove the following boundedness result for the Riesz transforms associated with $L$.
Let $M: L^{2}(\R^{n},d\gamma) \to L^{2}(\R^{n},d\gamma)$ be defined by $MH_{\alpha} = |\alpha|^{-\frac{1}{2}}H_{\alpha}$
for all $\alpha \in \Z^{n}_{+} \backslash \{0\}$, and $MH_{0} = 0$. 
\begin{theorem}
\label{thm:riesz}
For all $k=1,..,n$, the Riesz transforms 
$$
R_{k} = \partial_{x_{k}} M, \quad S_{k} = \partial^{*} _{x_{k}}M, 
$$
extend to bounded operators from $h^{1}(\gamma)$ to $L^{1}(\gamma)$.
\end{theorem}
Recall that $h^{1}(\gamma):=h^{1}_{\text{quad},2}(\gamma)$.
The proof of this theorem follows the approach of the preceding sections. We start with an appropriate Calder\'on reproducing formula, which can be established through chaos expansion. 
\begin{lemma}
For all $N \in \N$, $k \in \{1,...,n\}$, and $a,\alpha>0$, there exists $C>0$ such that for all $u \in L^{2}(\gamma)$
\begin{equation*}
\begin{split}
u &= C \int \limits _{0} ^{\infty} (t^{2}L)^{N+\frac{3}{2}}e^{\frac{5t^{2}}{\alpha}L}u \frac{dt}{t},\\
R_{k}u &= C \int \limits _{0} ^{\infty} t\partial_{x_{k}}(t^{2}L)^{N+1}e^{\frac{5t^{2}}{\alpha}L}u \frac{dt}{t}, 
\quad S_{k}u = C \int \limits _{0} ^{\infty} t\partial_{x_{k}}^{*}(t^{2}L)^{N+1}e^{\frac{5t^{2}}{\alpha}L}u \frac{dt}{t}.
\end{split}
\end{equation*}
\end{lemma}
In what follows, $k \in \{1,...,n\}$ is fixed. With the same proof as Corollary \ref{cor:dec}, we get the following.
\begin{corollary}
For all $N \in \N$, $b>0$, and $\alpha>4$, there exists $C>0$ and $n$ sequences of atoms  $(F_{m,j})_{m\in \N}$ and complex numbers 
$(\lambda_{m,j})_{m\in\N}$ for $j=1,...n$, such that for all 
$u \in C_{c}^{\infty}(\R^{n})$ and $x\in \R^{n}$:
\begin{equation*}
\begin{split}
-R_{k}u(x) = 
& C \sum \limits _{j=1} ^{n} \sum \limits _{m=1} ^{\infty} \lambda_{m,j} \int \limits _{0} ^{\frac{m(x)}{b}} 
t\partial_{k}(t^{2}L)^{N}e^{\frac{t^{2}}{\alpha}L}t\partial_{x_{j}}^{*}F_{m,j}(t,x)\frac{dt}{t}
\\&+C \sum \limits _{j=1} ^{n} \int \limits _{0} ^{\frac{m(x)}{b}} t\partial_{k}(t^{2}L)^{N}e^{\frac{t^{2}}{\alpha}L}t\partial_{x_{j}}^{*}(1_{D^{c}}(t,.)t\partial_{x_{j}}e^{\frac{4t^{2}}{\alpha}L})u(x)\frac{dt}{t}+ C \int \limits _{\frac{m(x)}{b}} ^{\infty} t\partial_{k}(t^{2}L)^{N+1}e^{\frac{5t^{2}}{\alpha}L}u(x) \frac{dt}{t},
\end{split}
\end{equation*}
and $ \sum \limits _{j=1} ^{n} \sum \limits _{m=1} ^{\infty} |\lambda_{m,j}| \lesssim \|u\|_{h^{1}_{\text{quad},2}}$. 
\end{corollary}
The same result holds for $S_{k}u$ (replacing $\partial_{x_{k}}$ by its adjoint). 
Theorem \ref{thm:riesz} will be proven, once we have obtained the following three estimates (and their analogues for $\partial_{x_{k}}^{*}$ instead of $\partial_{x_{k}}$).
$$\|\int \limits _{0} ^{\frac{m(.)}{b}} 
t\partial_{k}(t^{2}L)^{N}e^{\frac{t^{2}}{\alpha}L}t\partial_{x_{j}}^{*}F(t,.)\frac{dt}{t}\|_{L^{1}(\gamma)}
\lesssim 1,$$
for all $t^{1,2}(\gamma)$ atoms $F$.  
$$
\|\int \limits _{0} ^{\frac{m(.)}{b}} t\partial_{k}(t^{2}L)^{N}e^{\frac{t^{2}}{\alpha}L}t\partial_{x_{j}}^{*}(1_{D^{c}}(t,.)t\partial_{x_{j}}e^{\frac{4t^{2}}{\alpha}L})u\frac{dt}{t}\|_{L^{1}(\gamma)} \lesssim \|u\|_{L^{1}(\gamma)}.
$$
$$
\|\int \limits _{\frac{m(.)}{b}} ^{\infty} t\partial_{k}(t^{2}L)^{N+1}e^{\frac{5t^{2}}{\alpha}L}u\frac{dt}{t}
\|_{L^{1}(\gamma)} \lesssim \|u\|_{L^{1}(\gamma)}.
$$
We start with the relevant kernel estimate.
 \begin{lemma}
 \label{lem:RTest}
Let $N \in \Z_{+}$, $j \in \{1,...,n\}$,  and $\alpha \geq 4e^{8}$. Let $x,y \in \R^{n}$ and $t\in (0,a]$.  
If $t \lesssim m(y)$ then $|t\partial_{x_{k}}\tilde{K}_{t^{2},N,\alpha,j}(x,y)|\lesssim (1+t|x|)
 \exp(-\frac{\alpha}{4e^{8}}\frac{|e^{-t^{2}}y-x|^{2}}{1-e^{-2t^{2}}})M_{t^{2}}(x,y)$.
 \end{lemma}
 \begin{proof}
 As in Corollary \ref{cor:Ktilde}, there exists $C_{N}\in \N$ and two polynomials of $2n$ variables $Q_{N},\tilde{Q}_{N}$ of degree $C_{N}$ such that for all $x,y \in \R^{n}$ and $t>0$:
\begin{equation*}
\begin{split}
t\partial _{x_{k}}&\tilde{K}_{t^{2},N,\alpha,j}(x,y) = \\
&t^{2N+2}(1-e^{-\frac{2t^{2}}{\alpha}})^{-(N+1)}\tilde{Q}_{N}((\frac{e^{-\frac{t^{2}}{\alpha}}y_{j}-x_{j}}{\sqrt{1-e^{-\frac{2t^{2}}{\alpha}}}})_{j=1,...,n}  , (\sqrt{1-e^{-\frac{2t^{2}}{\alpha}}}y_{j})_{j=1,...,n})M_{\frac{t^{2}}{\alpha}}(y,x)
exp(|x|^{2}-|y|^{2})\\
+&t^{2N+2}x_{k}(1-e^{-\frac{2t^{2}}{\alpha}})^{-(N+\frac{1}{2})}Q_{N}((\frac{e^{-\frac{t^{2}}{\alpha}}y_{j}-x_{j}}{\sqrt{1-e^{-\frac{2t^{2}}{\alpha}}}})_{j=1,...,n}  , (\sqrt{1-e^{-\frac{2t^{2}}{\alpha}}}y_{j})_{j=1,...,n})M_{\frac{t^{2}}{\alpha}}(y,x)
exp(|x|^{2}-|y|^{2}).
\end{split}
\end{equation*}
Therefore
$$
|t\partial _{x_{k}}\tilde{K}_{t^{2},N,\alpha,j}(x,y)| \lesssim (1+t|x|)exp(-\frac{1}{2}
 \frac{|e^{-\frac{t^{2}}{\alpha}}y-x|^{2}}{1-e^{-\frac{2t^{2}}{\alpha}}})exp(|x|^{2}-|y|^{2}).
$$
Using Lemma \ref{lem:slow2}, and the fact that $t \lesssim m(y)$, we have that
\begin{equation*}
\begin{split}
|t\partial _{x_{k}}\tilde{K}_{t^{2},N,\alpha,j}(x,y)| &\lesssim (1+t|x|) \exp(-\frac{\alpha}{4e^{8}}\frac{|e^{-t^{2}}y-x|^{2}}{1-e^{-2t^{2}}})M_{t^{2}}(y,x)exp(|x|^{2}-|y|^{2})\\&= (1+t|x|)
\exp(-\frac{\alpha}{4e^{8}}\frac{|e^{-t^{2}}y-x|^{2}}{1-e^{-2t^{2}}})M_{t^{2}}(x,y). 
\end{split}
\end{equation*}
 \end{proof}

\begin{proposition}
Let $N \in \N$, $j\in \{1,...,n\}$ and $\alpha>2^{32}$. Let $B=B(c_{B},r_{B})\in \mathcal{B}_{2}$ and $F$ be a $t^{1,2}(\gamma)$ atom $F$ associated with $B$. 
\begin{enumerate}[(i)]
\item
$
\|\int \limits _{0} ^{r_{B}} |t\partial_{x_{k}}(t^{2}L)^{N}e^{\frac{t^{2}}{\alpha}L}t\partial_{x_{j}}^{*}F(t,.)|\frac{dt}{t}\|_{L^{1}(\gamma)} \lesssim 1,
$
\item
$
\|\int \limits _{0} ^{r_{B}} |t\partial_{x_{k}}^{*}(t^{2}L)^{N}e^{\frac{t^{2}}{\alpha}L}t\partial_{x_{j}}^{*}F(t,.)|\frac{dt}{t}\|_{L^{1}(\gamma)} \lesssim 1.
$
\end{enumerate}
\end{proposition}
\begin{proof}
For $l \in \Z_{+}$, we have, using Lemma \ref{lem:mm}:
\begin{equation*}
\begin{split}
\|1_{C_{l}(B)}\int \limits _{0} ^{r_{B}}|t\partial_{x_{k}}(t^{2}L)^{N}e^{\frac{t^{2}}{\alpha}L}t\partial_{x_{j}}^{*}F(t,.)|\frac{dt}{t}\|_{L^{1}(\gamma)} &\lesssim \sqrt{\gamma(2^{l+1}B)}\|1_{C_{l}(B)}\int \limits _{0} ^{r_{B}}|t\partial_{x_{k}}(t^{2}L)^{N}e^{\frac{t^{2}}{\alpha}L}t\partial_{x_{j}}^{*}F(t,.)|\frac{dt}{t}\|_{L^{2}(\gamma)}\\
&\lesssim 2^{2^{9}.4^{l}}\sqrt{\gamma(B)}\|1_{C_{l}(B)}\int \limits _{0} ^{r_{B}}|t\partial_{x_{k}}(t^{2}L)^{N}e^{\frac{t^{2}}{\alpha}L}t\partial_{x_{j}}^{*}F(t,.)|\frac{dt}{t}\|_{L^{2}(\gamma)}.
\end{split}
\end{equation*}
For $l=0$, we use the $L^2$ boundedness of $R_{j}$, and duality.
\begin{equation*}
\begin{split}
\|1_{C_{0}(B)}\int \limits _{0} ^{r_{B}}|t\partial_{x_{k}}(t^{2}L)^{N}e^{\frac{t^{2}}{\alpha}L}t\partial_{x_{j}}^{*}F(t,.)|\frac{dt}{t}\|_{L^{2}(\gamma)}
&\lesssim (\int \limits _{0} ^{r_{B}} \int \limits _{B} |F(t,x)|^{2} d\gamma(x)\frac{dt}{t})^{\frac{1}{2}}
\underset{\|g\|_{2}\leq 1}{\sup} (\int \limits _{0} ^{r_{B}} 
\|(t^{2}L)^{N+1}e^{\frac{t^{2}}{\alpha}L}R_{k}^{*}g\|_{L^{2}(\gamma)}^{2} \frac{dt}{t} )^{\frac{1}{2}} \\
&\lesssim \gamma(B)^{-\frac{1}{2}} \underset{\|g\|_{2}\leq 1}{\sup} \|R_{k}^{*}g\|_{L^{2}(\gamma)} \lesssim \gamma(B)^{-\frac{1}{2}},
\end{split}
\end{equation*}
where we have used chaos decomposition (or the $L^{2}$ functional calculus of $L$) as in the proof of Proposition \ref{prop:mol}.
For $l>0$, we use off-diagonal estimates, obtained from Lemma \ref{lem:RTest} as in Lemma \ref{lem:od}, and the fact that $|r_{B}x| \lesssim r_{B}|x-c_{B}|+1 \lesssim 2^{l}$ for all  $x \in C_{l}(B)$.
\begin{equation*}
\begin{split}
\|1_{C_{l}(B)}\int \limits _{0} ^{r_{B}}|t\partial_{x_{k}}(t^{2}L)^{N}e^{\frac{t^{2}}{\alpha}L}t\partial_{x_{j}}^{*}F(t,.)|\frac{dt}{t}\|_{L^{2}(\gamma)}
&\lesssim 2^{l}\int \limits _{0} ^{r_{B}} exp(-\frac{\alpha}{2^{6}e^{8}}4^{l}(\frac{r_{B}}{t})^{2})
\|F(t,.)\|_{L^{2}(\gamma)} \frac{dt}{t}\\
& \lesssim 2^{l} exp(-\frac{\alpha}{2^{23}}4^{l})
(\int \limits _{0} ^{1} exp(-\frac{\alpha}{2^{22}}4^{l}(\frac{1}{t})^{2})
 \frac{dt}{t})^{\frac{1}{2}} \gamma(B)^{-\frac{1}{2}}\\
& \lesssim 2^{l} exp(-\frac{\alpha}{2^{23}}4^{l})\gamma(B)^{-\frac{1}{2}}.
\end{split}
\end{equation*}
Summing in $l$ gives (i).\\
The same argument also gives 
$
\|x \mapsto \int \limits _{0} ^{r_{B}} |tx(t^{2}L)^{N}e^{\frac{t^{2}}{\alpha}L}t\partial_{x_{j}}^{*}F(t,.)|\frac{dt}{t}\|_{L^{1}(\gamma)} \lesssim 1,
$ and thus (ii).
\end{proof}
We now turn to the remainder terms. With exactly the same proof as Proposition \ref{prop:Dcomp} ,we get the following.
\begin{proposition}
Let $N \in \Z_{+}$, $j \in \{1,...,n\}$ and $\alpha > \max(32e^{4},8e^{8})$. Let $b \geq 
\max(2e,\sqrt{\frac{32e^{4}}{(\alpha-32e^{4})(1-e^{-\frac{8}{\alpha}})}})$.
Then 
\begin{enumerate}[(i)]
\item
$
\|\int \limits _{0} ^{\frac{m(.)}{b}} t\partial_{x_{k}}t^{2N+1}L^{N}e^{\frac{t^{2}}{\alpha}L}\partial_{x_{j}}^{*}(1_{D^{c}}(t,.)t\partial_{x_{j}}e^{\frac{4t^{2}}{\alpha}L})u \frac{dt}{t}\|_{L^{1}(\gamma)} \lesssim \|u\|_{L^{1}(\gamma)}. $
\item
$
\|\int \limits _{0} ^{\frac{m(.)}{b}}  t\partial_{x_{k}}^{*}t^{2N+1}L^{N}e^{\frac{t^{2}}{\alpha}L}\partial_{x_{j}}^{*}(1_{D^{c}}(t,.)t\partial_{x_{j}}e^{\frac{4t^{2}}{\alpha}L})u \frac{dt}{t}\|_{L^{1}(\gamma)} \lesssim \|u\|_{L^{1}(\gamma)}. $
\end{enumerate}
\end{proposition}
The final estimate is obtained as in Proposition \ref{prop:jinf}.
\begin{proposition}
Let $N \in \Z_{+}$, $b>0$, and $\alpha>4e^{8}$. For all $u\in C_{c}^{\infty}(\R^{n})$, we have
\begin{enumerate}[(i)]
\item
$
\| \int\limits _{\frac{m(.)}{b}} ^{\infty} t\partial_{x_{k}}(t^{2}L)^{N+1}e^{\frac{5t^{2}}{\alpha}L}u \frac{dt}{t}\|_{L^{1}(\gamma)} \lesssim \|u\|_{L^{1}(\gamma)}.
$
\item
$
\| \int\limits _{\frac{m(.)}{b}} ^{\infty} t\partial_{x_{k}}^{*}(t^{2}L)^{N+1}e^{\frac{5t^{2}}{\alpha}L}u \frac{dt}{t}\|_{L^{1}(\gamma)} \lesssim \|u\|_{L^{1}(\gamma)}.
$
\end{enumerate}
\end{proposition}

\begin{proof}
Let $M>0$ and $x \in \R^{n}$. Using Corollary \ref{cor:Ktilde} and Lemma \ref{lem:slow2}, we have that
\begin{equation*}
\begin{split}
|\int\limits _{\frac{m(x)}{b}} ^{M} & t\partial_{x_{k}}(t^{2}L)^{N+1}e^{\frac{5t^{2}}{\alpha}L}u \frac{dt}{t}|
\lesssim 
|\int\limits _{\frac{5m(x)^{2}}{b^{2}\alpha}} ^{\frac{5M^{2}}{\alpha}} s^{N+\frac{1}{2}}
\int \limits _{\R^{n}} \partial_{x_{k}}\partial_{s}^{N+1}M_{s}(x,y)u(y)dyds| \\
& \lesssim
\sum \limits _{l=0} ^{N}
\int \limits _{\R^{n}} Q_{l}(1,
(\frac{e^{-\frac{5m(x)^{2}}{b^{2}\alpha}}x_{j}-y_{j}}{\sqrt{1-e^{-2\frac{5m(x)^{2}}{b^{2}\alpha}}}})_{j=1,...,n}  , (\sqrt{1-e^{-2\frac{5m(x)^{2}}{b^{2}\alpha}}}x_{j})_{j=1,...,n})
M_{\frac{5m(x)^{2}}{b^{2}\alpha}}(x,y)u(y)|dy \\
& \qquad + \sum \limits _{l=0} ^{N} |M^{2l+1}\partial_{x_{k}}L^{l}e^{\frac{5}{\alpha}M^{2}L}u(x)|\\
& \lesssim
\int \limits _{\R^{n}} exp(-\frac{\alpha}{4e^{8}}
\frac{|e^{-\frac{5m(x)^{2}}{b^{2}}}x-y|^{2}}{1-e^{-2\frac{5m(x)^{2}}{b^{2}}}})
|u(y)|dy 
+\sum \limits _{l=0} ^{N} |M^{2l+1}\partial_{x_{k}}L^{l}e^{\frac{5}{\alpha}M^{2}L}u(x)|\\
&\lesssim \int \limits _{\R^{n}} M_{\frac{5m(x)^{2}}{b^{2}}}(x,y)|u(y)|dy 
+\sum \limits _{l=0} ^{N} |M^{2l+1}\partial_{x_{k}}L^{l}e^{\frac{5}{\alpha}M^{2}L}u(x)|.
\end{split}
\end{equation*}
Using chaos decomposition, this gives
$$
\| \int\limits _{\frac{m(.)}{b}} ^{M} t\partial_{x_{k}}(t^{2}L)^{N+1}e^{\frac{5t^{2}}{\alpha}L}u \frac{dt}{t}\|_{L^{1}(\gamma)} \lesssim
\int \limits_{\R^{n}} \int \limits _{\R^{n}} M_{\frac{5m(x)^{2}}{b^{2}}}(x,y)|u(y)|dy d\gamma(x)
+\sum \limits _{l=0} ^{N} M^{2l+1}e^{-\frac{5}{\alpha}M^{2}} \|u\|_{L^{2}(\gamma)},
$$
and thus, letting $M$ go to infinty
$$
\| \int\limits _{\frac{m(.)}{b}} ^{\infty} t\partial_{x_{k}}(t^{2}L)^{N+1}e^{\frac{5t^{2}}{\alpha}L}u \frac{dt}{t}\|_{L^{1}(\gamma)} \lesssim
\int \limits_{\R^{n}} \int \limits _{\R^{n}} M_{\frac{5m(x)^{2}}{b^{2}}}(x,y)|u(y)|dy d\gamma(x).
$$
The proof of \ref{prop:jinf} gives
$$
\int \limits_{\R^{n}} \int \limits _{\R^{n}} M_{\frac{5m(x)^{2}}{b^{2}}}(x,y)|u(y)|dy d\gamma(x)
\lesssim \|u\|_{L^{1}(\gamma)},
$$
which concludes the proof of (i). The same proof also gives (ii), using that $|xm(x)| \leq 1$ for all $x \in \R^{n}$.
\end{proof}

{\flushleft{\sc Pierre Portal}}\\
Permanent Address:\\
Universit\'e Lille 1, Laboratoire Paul Painlev\'e, F-59655 {\sc Villeneuve d'Ascq}.\\
Current Address:\\
Australian National University, Mathematical Sciences Institute, John Dedman Building, 
Acton ACT 0200, Australia.\\
{\tt pierre.portal@anu.edu.au}\\

\end{document}